\documentclass[a4paper,11pt]{article}

\usepackage{amsmath,amssymb,amsthm,amsfonts,graphics}
\usepackage{amsmath,amsfonts}
\usepackage{caption}
\usepackage{accents}
\usepackage{multirow}

\usepackage{multicol}
\usepackage{amsmath, amsfonts, amssymb}
\usepackage{amsthm}
\usepackage{graphicx}
\usepackage{pst-all,pst-infixplot,pst-math,pst-fractal,pst-solides3d}
\usepackage{pict2e}

\newtheorem{teo}{Theorem}[section]
\newtheorem{prop}{Proposition}[section]
\newtheorem{defi}{Definition}[section]

\newtheorem{lemma}{Lemma}[section]

\newtheorem{example}{Example}[section]
\newtheorem{assumption}{Assumption}[section]
\newtheorem{remark}{Remark}[section]

\newcommand{\be}{\begin{eqnarray*}}
	\newcommand{\ben}{\begin{eqnarray}}
		\newcommand{\ee}{\end{eqnarray*}}
	\newcommand{\een}{\end{eqnarray}}
\newcommand{\al}{\begin{align*}}
	\newcommand{\eal}{\end{align*}}
\newcommand{\aln}{\begin{align}}
	\newcommand{\ealn}{\end{align}}

\newcommand{\fcal}{\mathcal{F}}

\DeclareMathOperator*{\supp}{supp}

\def\Om{\Omega}

\let\<\langle
\let\>\rangle
\let\phi\varphi

\usepackage{lineno,hyperref}
\modulolinenumbers[5]

\begin{document}


\title{Optimal Investment Decision Under Switching Regimes of Subsidy Support
}


\author{
Carlos Oliveira\thanks{Department of Mathematics and
CEMAT, Instituto Superior T\'ecnico, Universidade de
Lisboa, Av. Rovisco Pais, 1049-001 Lisboa, Portugal,
Email: carlosmoliveira@tecnico.ulisboa.pt}\and
Nicolas Perkowski\thanks{Institut f\"ur Mathematik, Humboldt-Universit\"at zu Berlin,
Unter den Linden 6, D-10099 Berlin,
Germany}}

\title{\sc Optimal Investment Decision Under Switching regimes of Subsidy Support}

\maketitle

\begin{abstract}
\noindent
We address the problem of making a managerial decision when the investment project is subsidized, which results in the resolution of an infinite-horizon optimal stopping problem of a switching diffusion driven by either an homogeneous or an inhomogeneous continuous-time Markov chain. 

We provide a characterization of the value function (and optimal strategy) of the optimal stopping problem. On the one hand, broadly, we can prove that the value function is the unique viscosity solution to a system of HJB equations. On the other hand, when the Markov chain is homogeneous and the switching diffusion is one-dimensional, we obtain stronger results: the value function is the difference between two convex functions.
\mbox{}

\end{abstract}

\section{Introduction}
The optimal time to make managerial decisions has been broadly studied in the context of Real Options since the pioneering works of Dixit and Pindyck \cite{dixit1994investment} and Trigeorgis \cite{trigeorgis1996real}. Over time, while trying to fit the market's necessities, this type of models has become more and more complex from both the economic and the mathematical point of view. From the economic side, the number of sequential decisions studied in these models has increased and, from the mathematical angle, the associated stochastic control problems have become progressively more difficult to solve. 

In the past few years, several authors have introduced in real options models the existence of temporary subsidy support schemes in order to study their influence in the optimal investment time. This is particularly important in subsidized fields such as renewable energies, where there is an intense research activity (see, for instance, Boomsma, Meade and Fleten \cite{boomsma2012renewable}), Boomsma and Linnerud \cite{boomsma2015market}, Adkins and Paxson \cite{adkins2016subsidies}, Fleten, Linnerud, Moln{\'a}r and Nygaard \cite{fleten2016green}, Kitzing, Juul, Drud, Boomsma \cite{kitzing2017real} and Guerra, Kort, Nunes and Oliveira \cite{guerra2017hysteresis}).

Following the previously cited authors, we formulate an investment model in a more general sense, where we assume that: (1) there are various different levels of subsidy, (2) the coefficients of the dynamic relative to the economic indicator change with the level of subsidy and (3) the follow-up of the firm's situation is influenced by the time since the previous evaluation.  In consequence, we formulate our model as an infinite-horizon optimal stopping problem where the uncertainty is generally modeled by a switching diffusion driven by an inhomogeneous continuous-time Markov chain. 

There are a few articles on optimal stopping problems for switching diffusions, covering different topics of financial mathematics. On the one hand, Eloe, Liu, Yatsuki, Yin and Zhang \cite{eloe2008optimal}, Guo \cite{guo2001explicit}, and Guo and Zhang \cite{guo2004closed} give explicit solutions for a few particular problems; on the other hand, Pemy \cite{pemy2014optimal}, Pemy and Zhang \cite{pemy2006optimal}, and Liu \cite{liu2016optimal} show that, in certain conditions, the value function for the correspondent optimal stopping problem is a viscosity solution to a system of Hamilton-Jacobi-Bellman (HJB) equations. Very recently, Egami and Kevkhishvili \cite{egami2017optimal} show that these type of problems can be reduced to a set of optimal stopping problems without a switching regime.

In this work, we show that, in general, the value function is time-dependent and the unique viscosity solution to a system of HJB equations. Additionally, when the continuous-time Markov chain is homogeneous and when the diffusion is one-dimensional, the value function is the difference of two convex functions and the time-dependence is lost.

We organize the text as follows: in Section \ref{The stochastic process}, we describe the stochastic process that we consider; in Section \ref{Optimal stopping problem}, we define the optimal stopping problem and some of the required assumptions; in Section \ref{HJB equations} we prove that the value function is the unique viscosity solution to a system of HJB equations and, finally, in Section \ref{The One Dimensional Case}, we discuss the optimal stopping problem in the homogeneous and one-dimensional case.



\section{The stochastic process}\label{The stochastic process}
We consider an investment project enrolled in an assistant program where there are $k$ different levels of subsidy.  The process $\theta=\{\theta_s:\, s\geq 0\}$, which provides the information concerning the level of subsidy at the current moment, is such that
\begin{align}\label{theta}
	\begin{cases}
		&\theta_s\in \Theta\equiv\{1,\ldots,k\}, \quad \text{for each }s\geq 0,\\
		&\theta \text{ is a c\`adl\`ag process.}\\
	\end{cases}
\end{align}
To completely characterize the Markov chain $\theta$, we introduce the process $\{\nu_n\,:\,n\in\mathbb{N}_0\}$, where $\nu_n$, with $n\in\mathbb{N}$, is the time until the $n^{\text{th}}-$transition of Markov chain $\theta$, defined by
$$\nu_1=\inf\{s>0\,:\,\theta_{s^-}\neq\theta_s\}\quad\text{and}\quad\nu_n=\inf\{s>\nu_{n-1}\,:\,\theta_{s^-}\neq\theta_s\}.$$
We assume that for every $j,m\in\Theta$
\begin{align*}
	&P(\nu_n-\nu_{n-1}\leq s\,\vert\, \theta_{\nu_{n-1}}=j)=1-e^{-\int_0^s\lambda_j(u)du},\quad\text{for all }s\geq 0,\\
	&P(\theta_{\nu_n}=m\,\vert\,\theta_{\nu_{n-1}}=j)=p_{j,m}(\nu_{n}-\nu_{n-1}),
\end{align*}
where $\lambda_j:[0,\infty)\to[0,\infty)$ is continuous and  $p_{j,k}\,:\,[0,\infty)\to[0,1]$ is a continuous function such that $\sum_{m=1}^kp_{j,m}(s)=1$ and $p_{j,j}(s)=0$, for all $s>0$.  Additionally, we consider that, for every $n_1\neq n_2\in\mathbb{N}$, the random variables $(\nu_{n_1}-\nu_{n_1-1})\text{ and } (\nu_{n_2}-\nu_{{n_2}-1})$ are independent.

The investment project operates in a random environment characterized by an economic indicator, which is modeled by a  $n-$dimensional stochastic process  $X=\{X_s:s\geq 0\}$. This process solves the switching stochastic differential equation (SDE)
\begin{equation}\label{Eq SDE}
	dX_s=\alpha(X_s,\theta_s)ds+\sigma(X_s,\theta_s)dW_s,
\end{equation}
taking values in the open set $D\subseteq\mathbb{R}^n$, where $W=\{W_s,s\geq 0\}$ is an $m-$dimensional Brownian motion independent of $\theta$ and where $\alpha:D\to\mathbb{R}^n$ and $\sigma:D\to\mathbb{R}^{n\times m}$ are Borel measurable functions. Therefore, we build this model on a complete filtered probability space $(\Om, \fcal, (\fcal_s)_{s\geq 0}, P)$ satisfying the usual conditions and supporting the independent processes $\theta$ and $W$.

The next assumption characterizes the solution of the switching SDE \eqref{Eq SDE}. Some results concerning the existence and uniqueness of solutions to switching diffusions may be found in Mao and Yuan \cite{mao2006stochastic}, and Yin and Zhu \cite{yin2010hybrid}. Additionally, in Kallenberg \cite{kallenberg2006foundations}, Karatzas and Shreve \cite{karatzas2012brownian} and Krylov \cite{krylov2008controlled}, one can find results concerning the existence and uniqueness of SDEs without switching regimes. 
\begin{assumption}\label{Assumption 1}
	The Borel measurable functions $\alpha:D\times\Theta\to\mathbb{R}^n$ and $\sigma:D\times\Theta\to\mathbb{R}^{n\times m}$ are 
	such that the SDE \eqref{Eq SDE}, for each initial condition, has a unique strong solution $(W,X)$ on the filtered probability space $({\Om}, {\fcal}, ({\fcal}_s)_{s\geq 0}, {P})$ that remains in $D$ for all times.
	Additionally, we assume that 
	\begin{align*}
		\vert \alpha(x,i)-\alpha(y,i)\vert+\Vert\sigma(x,i)-\sigma(y,i)\Vert\leq L\vert x-y\vert.
	\end{align*}
\end{assumption}

For any set $I\subseteq D$ we define the ${\cal F}_s-$stopping time 
\begin{equation*}
	T^I\equiv\inf\{s\geq 0:X_s\in\partial I\} \text{ with } X_0=x\in I,
\end{equation*}
where  $\partial I$ is obtained by considering the topology on $D$,
which is the trace of the usual topology on $\mathbb{R}^n$. If $I=D$, then $\partial I=\emptyset$, since $D$ is open in the usual topology on $\mathbb{R}^n$.
In addition, we assume that $D$ is such that $P(T^A<\infty)>0$, for all open $A\subsetneq D$ and $x\in A$.

The process $(X,\theta)$ is not, in general, a Markov process, because it is never known how much time was spent since the last transition in the Markov chain. Therefore, we introduce the process $\zeta=\{\zeta_s,\,s\geq t\}$, which represents the time spent from the last change in the level of subsidy until the moment $s$, defined by 
\begin{equation*}\label{process zeta}
	\zeta_s=s-\nu^s,\,s>0,
\end{equation*}
where
$\nu^{s}\equiv\sup\{\nu_n\,:\nu_n\leq s, n\in\mathbb{N}\}$ is an ${\cal F}_s-$stopping time. 
Unless otherwise stated, we will work with the process $({X},\theta,\zeta)$, which is the Markovian representation of the process $({X},\theta)$.

\section{Optimal stopping problem}\label{Optimal stopping problem}
In this section, we formulate the stochastic optimization problem that we are interested in. Thus, we consider that the cash-flow associated with the investment project is different in the $k$ different levels of subsidy. Therefore, the running payoff is represented by the function $\Pi:D\times\Theta\to\mathbb{R}$, and the cost of abandonment is represented by $h:D\times\Theta\to\mathbb{R}$. Additionally, we represent the instantaneous interest rate with 
$r:D\times\Theta\to\mathbb{R}$.
\begin{assumption}\label{A functions}
	The functions $\Pi,h,r:D\times\Theta\to\mathbb{R}$ are such that 
	\begin{align}\nonumber
		&\Pi(\cdot,i),h(\cdot,i),r(\cdot,i)\in C( D),\quad \text{for all }i\in\Theta\\\nonumber
		&\exists\epsilon_i>0\text{ such that } r(\cdot,i)>\epsilon_i,\quad \text{for all }i\in\Theta.
	\end{align}
\end{assumption}
If the investment project is permanently abandoned at the moment $\tau$, where $\tau$ is a ${\cal F}_s-$stopping time,
its revenue is given by 
\begin{align*}
	&\int_0^\tau e^{-\rho_s}\Pi(X_s,\theta_s)ds-e^{-\rho_\tau}h(X_\tau,\theta_\tau){1_{\{\tau< \infty\}}},\\
	&\rho_s=\int_0^{s}r(X_u,\theta_u)du, \quad s\geq 0.
\end{align*}
Therefore, the expected outcome associated with the project, when the initial observation is $(X_0,\zeta_0,\theta_0)=(x,t,i)$,  is given by the functional 
\begin{align}\label{Eq expected outcome}
	J(x,t,i,\tau)&=E_{x,t,i}\left[\int_0^\tau e^{-\rho_s}\Pi(X_s,\theta_s)ds-e^{-\rho_\tau}h(X_\tau,\theta_\tau){1_{\{\tau< \infty\}}}\right].
\end{align}
Here, $E_{x,t,i}\left[\cdot\right]$ is the expected value conditional on $X_0=x$, $\zeta_0=t$ and $\theta_0=i$\footnote{Throughout the paper, we also use the notation $E_{x,i}[\cdot]$ (resp., $E_{t,i}[\cdot]$) representing the expected value conditional on $X_0=x$ and $\theta_0=i$ (resp., $\zeta_0=t$ and $\theta_0=i$).}.
Our main goal is to seek the ${\cal F}_s-$stopping time, $\tau^*$, maximizing the
expected outcome \eqref{Eq expected outcome} in a certain open and connected set $I$, which should satisfy the following property: the set $\partial I$ is regular for the process ${X}$ in the sense that, 
\begin{equation*}
	T^I=0,\text{ }P-\text{almost surely, for all }x\in \partial I.
\end{equation*} 
Notice that, in this formulation, the project is necessarily abandoned for $s> T^I$. Therefore, if ${\cal T}$ is the set of all ${\cal F}_s-$stopping times and ${\cal S}=\{\tau\wedge T^I:\tau\in{\cal T}\}$, we intend to find the value function $V^*$, verifying  
\begin{equation}\label{Optimal Stopping Problem}
	V^*(x,t,i)=\sup_{\tau\in {\cal S}} J(x,t,i,\tau), \quad (x,t,i)\in \overline{I}\times[0,\infty)\times\Theta.\footnote{From now on, we use the following notation: $\overline I=I\cup\partial I$.}
\end{equation}
Since the strategy $\tau\equiv 0$ (to stop immediately, regardless of the current state $(X_0,\zeta_0,\theta_0)$) verifies $J(x,t,i,0)=-h(x,i)$, it is obvious that $V^*\geq -h$. Thus, an optimal stopping time is given by the rule
\begin{align*}
	&\tau^*=\inf\{s\geq 0:V^*(X_s,\zeta_{s},\theta_s)\leq-h(X_s,\theta_s)\}.
\end{align*}

In what follows, for every real function $f$, we set $f^+=\max(0,f)$, $f^-=\max(0,-f)$. Thus $f=f^+-f^-$ and $\vert f\vert=f^++f^-$. The problem's well-posedness is guaranteed by introducing the following integrability conditions:
\begin{assumption}\label{Assumption : last one}
	The functions $\Pi,h,r:D\times\Theta\to\mathbb{R}$ are such that 
	\begin{align*}
		&E_{x,t,i}\left[\int_0^{T^I} e^{-\rho_s}\Pi^+({X}_t,\theta_t)ds\right]<\infty\quad and \\\label{uniform integrability eh}
		&\{ h(X_\tau,\theta_\tau)\}_{\tau\in {\cal S}}\quad\text{is a uniformly integrable family of random variables. }
	\end{align*}
\end{assumption}
For future reference, we notice that according to Assumption \ref{Assumption : last one}, for any initial condition $(X_0,\zeta_0,\theta_0)=(x,t,i)$,
\begin{equation*}\label{UI1}
	\left\{\int_0^\tau e^{-\rho_s}\Pi^+({X}_t,\theta_t)ds-e^{-\rho_\tau}h(X_\tau,\theta_\tau)\right\}_{\tau\in{\cal S}}
\end{equation*}
is a uniformly integrable family of random variables, meaning that there is a uniform integrability test function $f:[0,\infty)\to[0,\infty)$ (see Definition C.2 and Theorem C.3 in {\O}ksendal \cite{oksendal1998stochastic}) such that 
\begin{equation}\label{UI3}
	\sup_{\tau\in{\cal S}}E_{x,t,i}\left[f\left(\left\vert\int_0^\tau e^{-\rho_s}\Pi^+({X}_t,\theta_t)ds-e^{-\rho_\tau}h(X_\tau,\theta_\tau)\right\vert\right)\right]<\infty.
\end{equation}
To finalize this section, in the next proposition we establish that under the assumptions considered in this section, $\{V^*(X_\tau,\zeta_{\tau},\theta_\tau)\}_{\tau\in{\cal S}}$ is a uniformly integrable family of random variables.
\begin{prop}
	Let $V^*$ be the value function defined as in \eqref{Optimal Stopping Problem}. Then, $\{V^*(X_\tau,\zeta_{\tau},\theta_\tau)\}_{\tau\in{\cal S}}$ is a uniformly integrable family of random variables, for every initial condition $(X_0,\zeta_0,\theta_0)=(x,t,i)$.
\end{prop}
\begin{proof}
	We start by noting that, by definition,
	\begin{align*}
		V^*(x,t,i)&=E_{x,t,i}\left[ \int_0^{\tau^*} e^{-\rho_s}\Pi({X}_s,\theta_s)ds-e^{-\rho_{\tau^*}}h(X_{\tau^*},\theta_{\tau^*})\right]\\
		&\leq E_{x,t,i}\left[ \int_0^{\tau^*} e^{-\rho_s}\Pi^+({X}_s,\theta_s)ds-e^{-\rho_{\tau^*}}h(X_{\tau^*},\theta_{\tau^*})\right].
	\end{align*}
	Consequently, choosing a function $f:[0,\infty)\to[0,\infty)$ as in \eqref{UI3}, which is convex, we get, for any $\tau\in{\cal S}$,
	\begin{align}
		\nonumber E_{x,t,i}&\left[f\left(\left\vert V^*(X_\tau,\zeta_\tau,\theta_\tau)\right\vert\right)\right]\leq \\	\nonumber
		&~~~~~~~~~~~~~\leq E_{x,t,i}\left[f\left(\left\vert E_{X_\tau,\zeta_\tau,\theta_\tau}\left[ \int_0^{\tau^*} e^{-\rho_s}\Pi^+({X}_s,\theta_s)ds-e^{-\rho_{\tau^*}}h(X_{\tau^*},\theta_{\tau^*})\right]\right\vert\right)\right]\\\label{UI auxiliary}
		&~~~~~~~~~~~~~\leq E_{x,t,i}\left[f\left(\left\vert\int_0^{\tau^*} e^{-\rho_s}\Pi^+(X_s,\theta_s)ds-e^{-\rho_{\tau^*}}h(X_{\tau^*},\theta_{\tau^*})\right\vert\right)\right]<\infty.
	\end{align}
	The first inequality in \eqref{UI auxiliary} follows from the strong Markov property and the Jensen's inequality, while the second inequality follows from Equation \eqref{UI3}. 
\end{proof}

\section{HJB equations}\label{HJB equations}
In this section, our main goal is to provide the system of HJB equations associated with the optimal stopping problem \eqref{Optimal Stopping Problem}. Furthermore, we will prove that, under certain conditions, the value function $V^*$ is the unique viscosity solution to this system of HJB equations. 

In Section \ref{Dynamic programming principle}, we present a weak version of the dynamic programming principle (DPP) that we will use in the following sections. A general formulation of this DPP can be found in Bouchard and Touzi \cite{bouchard2011weak}.
\subsection{Dynamic programming principle}\label{Dynamic programming principle}
Consider the Markov process $({X},\theta,\zeta)$, and its infinitesimal generator ${\cal L}$, defined by 
\begin{equation}\label{Eq infinitesimal generator}
	({\cal L}\phi)(x,t,i)=\lim\limits_{u\downarrow 0}\frac{1}{u}E_{x,t,i}\left[\phi({X}_{u},\zeta_{u},{\theta_{u}})-\phi(x,t,i)\right],
\end{equation}
for all $\phi$ in the domain of $\cal L$. In the next proposition, we present an expression for ${\cal L}$. 
\begin{prop}\label{P infinitesimal generator}
	Let $({X},\theta,\zeta)$ be the $(n+1+1)$-dimensional process defined as in Section \ref{The stochastic process}. Then, the infinitesimal generator ${\cal L}$ defined in \eqref{Eq infinitesimal generator} is such that
	\begin{align*}
		({\cal L}\phi)(x,t,i)&=\frac{\partial\phi}{\partial t}(x,i,t)+\alpha(x,i)D{\phi}(x,t,i) +\frac{1}{2}Tr\left[\sigma\sigma^T(x,i)D^2{\phi}(x,t,i)\right]\\\nonumber
		&+\left(Q\phi\right)(x,t,i)\\
		\left(Q\phi\right)(x,t,i)&=\sum_{j\neq i}\lambda_{i,j}(t)\left(\phi(x,{0},j)-{\phi}(x,t,i)\right),\quad \text{for a fixed }i\in\Theta,
	\end{align*}
	where $\lambda_{i,j}(t)=p_{i,j}(t)\lambda_i(t)$, for every $t\geq 0$ and $\phi:\mathbb{R}^n\times\Theta\times[0,\infty)\to\mathbb{R}$ is such that $\phi(\cdot,\cdot,i)\in C^{2,1}_0(D\times[0,\infty))$\footnote{A function $\phi\in C^{2,1}_0(I\times[0,\infty))$ (resp., $\phi\in C^{2}_0(I\times[0,\infty))$) if $\phi\in C^{2,1}(I\times[0,\infty))$ (resp., $\phi\in C^{2}(I\times[0,\infty))$) and has compact support. }, for $i\in\Theta$. 
\end{prop}

Before we prove Proposition \ref{P infinitesimal generator}, we note that the process $({X},\zeta,\theta)$ is a semimartingale (indeed $X$ is the sum of a martingale and a finite variation process, and $\zeta$ and $\theta$ are finite variation processes) and, consequently, admits a generalized It\^o decomposition (see Theorem II.33 from Protter \cite{protter1990stochastic}). Indeed, for any function $\phi:D\times[0,\infty)\times\Theta\to\mathbb{R}$ such that $\phi(\cdot,\cdot,i)\in C^{2,1}({D}\times[0,\infty))$ and any $u> 0$,
\begin{align}\nonumber
	\phi({X}_u,\zeta_{u},\theta_u)-\phi(x,t,i)&=\int_0^{u}\alpha(X_s,\theta_s)\cdot D{\phi}(X_s,\zeta_{s},\theta_s)+\frac{1}{2}Tr\left[\sigma\sigma^T(X_s,\theta_s)D^2{\phi}(X_s,\zeta_{s},\theta_s)\right]ds\nonumber\\\nonumber
	&+\int_{0}^u\frac{\partial\phi}{\partial t}(X_s,\zeta_{s},\theta_s)ds+\int_0^{u}D{\phi}(X_s,\zeta_{s},\theta_s)\sigma(X_s,\theta_s)dW_s\\
	&+\sum_{n\in\mathbb{N}}\left(\phi\left({X}_{{\nu_n}},0,\theta_{\nu_n}\right)-{\phi}\left(X_{{\nu_n^-}},\zeta_{\nu_n^-},\theta_{\nu_n^-}\right)\right)1_{\{u\geq\nu_n\}}.\label{ito formula}
\end{align}
\begin{proof}[Proof of Proposition \ref{P infinitesimal generator}] 
	Taking into account the It\^o formula presented in Equation \eqref{ito formula}, we get that	
	\begin{align*}
		E_{x,t,i}&\left[\phi({X}_{u},\zeta_{u},{\theta_{u}})-\phi(x,t,i)\right]=
		\\&=E_{x,t,i}\Bigg[\int_0^{u}\alpha(X_s,\theta_s)\cdot D{\phi}(X_s,\zeta_{s},\theta_s)+\frac{1}{2}Tr\left[\sigma\sigma^T(X_s,\theta_s)D^2{\phi}(X_s,\zeta_{s},\theta_s)\right]ds\nonumber\\
		&+\sum_{n\in\mathbb{N}}\left(\phi\left({X}_{{\nu_n}},0,\theta_{\nu_n}\right)-{\phi}\left(X_{{\nu_n^-}},\zeta_{\nu_n^-},\theta_{\nu_n^-}\right)\right)1_{\{u\geq\nu_n\}}\Bigg].
	\end{align*}
	Furthermore, we note that
	\begin{align}\label{jumps decomposition}
		E_{x,t,i}&\Bigg[\sum_{n\in\mathbb{N}}\left(\phi\left({X}_{{\nu_n}},0,\theta_{\nu_n}\right)-{\phi}\left(X_{{\nu_n^-}},\zeta_{\nu_n^-},\theta_{\nu_n^-}\right)\right)1_{\{u\geq\nu_n\}}\Bigg]=\\\label{jumps decomposition second line}
		=&E_{x,t,i}\Bigg[\left(\phi\left({X}_{{\nu_1}},0,\theta_{\nu_1}\right)-{\phi}\left(X_{{\nu_1^-}},\zeta_{\nu_1^-},\theta_{\nu_1^-}\right)\right)1_{\{u\geq\nu_1\}}\Bigg]\\
		\nonumber
		&+E_{x,t,i}\left[\sum_{n>1}E_{x,t,i}\Bigg[\left(\phi\left({X}_{{\nu_n}},0,\theta_{\nu_n}\right)-{\phi}\left(X_{{\nu_n^-}},\zeta_{\nu_n^-},\theta_{\nu_n^-}\right)\right)1_{\{u\geq\nu_n>\nu_{n-1}\}}\left\vert ({\nu_{n-1}},\theta_{\nu_{n-1}})\right.\right]\Bigg]
	\end{align}
	where, for $n>1$,
	\begin{align}\label{auxiliary jump}
		&E_{x,t,i}\left[\left(\phi\left({X}_{\nu_{n}},0,\theta_{\nu_{n}}\right)-{\phi}\left({X}_{\nu_{n}^-},\zeta_{\nu_{n}^-},\theta_{\nu_{n}^-}\right)\right)1_{\{u\geq\nu_n>\nu_{n-1}\}}\vert({\nu_{n-1}},\theta_{\nu_{n-1}})\right]=\\\nonumber
		&=\sum_{m=1}^{k}	E_{x,t,i}\left[\left(\phi\left({X}_{\nu_{n}},0,m\right)-{\phi}\left({X}_{\nu_{n}^-},\zeta_{\nu_{n}^-},\theta_{\nu_{n-1}}\right)\right)p_{\theta_{\nu_{n-1}},m}(\nu_n-\nu_{n-1})1_{\{u\geq\nu_n\}}\vert({\nu_{n-1}},\theta_{\nu_{n-1}})\right].
	\end{align}
	Since, for $n>1$, 
	\begin{align*}
		P\left(\nu_n<s\vert\left({\nu_{n-1}},\theta_{\nu_{n-1}}\right)\right)&=\begin{cases}
			1-e^{-\int_{\nu_{n-1}}^s\lambda_{\theta_{\nu_{n-1}}}(\omega-\nu_{n-1})d\omega},&\text{if }s>\nu_{n-1}\\
			0,&\text{if }s\leq\nu_{n-1}
		\end{cases}.
	\end{align*}
	Equation \eqref{auxiliary jump} can be given by
	\begin{align*}
		&\nonumber
		E_{x,t,i}\left[	\int_{\nu_{n-1}\wedge u}^uf_{\theta_{\nu_{n-1}}}(s)\sum_{m=1}^{k}\left(\phi\left({X}_{s},0,m\right)-{\phi}\left(X_{s},s,\theta_{\nu_{n-1}}\right)\right)p_{\theta_{\nu_{n-1}},m}(\nu_n-\nu_{n-1})ds\vert({\nu_{n-1}},\theta_{\nu_{n-1}})\right],
	\end{align*}
	where $f_{\theta_{\nu_{n-1}}}(s)=\lambda_{\theta_{\nu_{n-1}}}(s-\nu_{n-1})e^{\int_{\nu_{n-1}}^s\lambda_{\theta_{\nu_{n-1}}}(\omega-\nu_{n-1})d\omega}$, for $s>\nu_{n-1}$. Furthermore, as $E[1_{\{\nu_n\geq s\}}\vert({\nu_{n-1}},\theta_{\nu_{n-1}})]=e^{\int_{\nu_{n-1}}^s\lambda_{\theta_{\nu_{n-1}}}(\omega-\nu_{n-1})d\omega}$, for $s>\nu_{n-1}$,  Equation \eqref{auxiliary jump} can also be given by
	\begin{align}
		&\nonumber
		E_{x,t,i}\Bigg[	\int_{\nu_{n-1}\wedge u}^{\nu_n\wedge u}\lambda_{\theta_{\nu_{n-1}}}(s-\nu_{n-1})p_{\theta_{\nu_{n-1}},m}(\nu_n-\nu_{n-1})\times\\\label{jumps decomposition last part}
		&~~~~~~~~~~~~~~~~~~~~~~~~~~~~\times\sum_{m=1}^{k}\left(\phi\left({X}_{s},0,m\right)-{\phi}\left(X_{s},s,\theta_{\nu_{n-1}}\right)\right)ds\vert({\nu_{n-1}},\theta_{\nu_{n-1}})\Bigg].
	\end{align}
	A similar representation can be found for the expected value in \eqref{jumps decomposition second line}, since we have
	\begin{equation}\label{probability nu1}
		P\left(\nu_1<s\vert\left(\zeta_0,\theta_0\right)=\left(t,i\right)\right)=\begin{cases}
			1-e^{-\int_{0}^{s+t}\lambda_{{i}}(\omega)d\omega},&\text{if }s>0\\
			0,&\text{if }s\leq 0
		\end{cases}.
	\end{equation}
	Then, by using the definition of the infinitesimal generator in \eqref{Eq infinitesimal generator}, the result is straightforward.
\end{proof}
It can be useful to consider the operator $\tilde{\cal L}$ given by
\begin{align*}
	(\tilde{\cal L}\phi)(x,t,i)&=\lim\limits_{h \downarrow 0}\frac{1}{h}E_{x,t,i}\left[e^{-\rho_h}\phi({X}_h,\zeta_h,{{\theta_h}})-\phi(x,t,i)\right]\\
	&=-r(x,i)\phi(x,t,i)+({\cal L}\phi)(x,t,i).
\end{align*}
For future reference, we note that along the same lines of Proposition \ref{P infinitesimal generator}, we can prove that the Dynkin's formula, for the $(n+1+1)-$dimensional process $({X},\zeta,\theta)$, holds true and verifies
\begin{align*}
	E_{x,t,i}\left[e^{-\rho_h}\phi({X}_{u},\zeta_{u},{\theta_{u}})\right]&=\phi({x},t,i)+E_{x,t,i}\left[\int_t^{u}e^{-\rho_s}(\tilde{\cal L}\phi)(X_s,\zeta_s,\theta_s) ds\right].
\end{align*}

In Proposition \ref{DPP}, a weak version of the DPP for the optimal stopping problem \eqref{Optimal Stopping Problem} is presented. The proof relies on the Markov structure of the process $(X,\zeta,\theta)$ and we follow the exposition of Guerra \cite{ehrhardt2017novel} (pages 143-167) and Touzi \cite{touzi2012optimal}. Before we introduce the DPP, we state an auxiliary result concerning the continuity of the function $(x,t)\to J(x,t,i,\tau)$. If necessary, to highlight the dependence of $X_s$ on the initial condition $X_0=x$ and on the element $\omega\in\Omega$, we will write $X_s^x$ and $X_s^x(\omega)$, respectively.

\begin{lemma}\label{L continuity of J}
	The function $(x,t)\to J(x,t,i,\tau)$ is continuous, for every $\tau\in{\cal S}$ and $i\in\Theta$. 
\end{lemma}
\begin{proof}
	Firstly, by definition of a solution to a switching SDE, the function $s\to X_s^x$ is continuous. Additionally, we prove that the function $x\to X_s^x$ is $P-$almost surely continuous. To do this, we note that 
	\begin{align*}
		\vert X_s^x-X_s^{x'}\vert^2
		&\leq 3\vert x-x'\vert^2+3Ls\int_0^s\vert X_u^x-X_u^{x'}\vert^2 ds+3\left\vert\int_0^s \sigma(X_u^x,\theta_u)-\sigma(X_u^{x'},\theta_u) dW_u\right\vert^2.
	\end{align*}
	From the Doob's maximal inequality, we get that,  for all $s'>0$
	$$
	E_{t,i}\left[\sup_{0\leq s\leq s'}\left\vert\int_0^s \sigma(X_u^x,\theta_u)-\sigma(X_u^{x'},\theta_u) dW_u\right\vert^2\right]\leq 4E_{t,i}\left[\left\vert\int_0^{s'} \sigma(X_u^x,\theta_u)-\sigma(X_u^{x'},\theta_u) dW_u\right\vert^2\right].
	$$
	Therefore, by using the It\^o isometry,
	\begin{align*}
		E_{t,i}\left[\sup_{0\leq s\leq s'}\vert X_s^x-X_s^{x'}\vert^2\right]
		&\leq 3\vert x-x'\vert^2+3Ls'E_{t,i}\left[\int_0^{s'}\vert X_u^x-X_u^{x'}\vert^2 du\right]\\
		&+12LE_{t,i}\left[\int_0^{s'} \left\vert X_u^x-X_u^{x'} \right\vert^2du\right]\\
		&\leq 3\vert x-x'\vert^2+(3Ls'+12L)\int_0^{s'}E_{t,i}\left[\vert X_u^x-X_u^{x'}\vert^2 \right]du\\
		&\leq 3\vert x-x'\vert^2+(3Ls'+12L)\int_0^{s'}E_{t,i}\left[\sup_{0\leq s\leq u}\vert X_s^x-X_s^{x'}\vert^2 \right]du.
	\end{align*}
	By the Gr\"onwall's inequality, we get 
	$$E_{t,i}\left[\sup_{0\leq s\leq s'}\vert X_s^x-X_s^{x'}\vert^2\right]\leq3\vert x-x'\vert^2e^{3Ls'+12L}, $$
	which proves the first statement  after an application of Kolmogorov's continuity criterion.
	
	To show that $(x,t)\to J(x,t,i,\tau)$ is continuous, we note that
	\begin{align*}
		&J(x,t,i,\tau)=E_{x,t,i}\left[\int_0^{\tau} e^{-\rho_s}\Pi(X_s,\theta_s)ds-e^{-\rho_\tau}h(X_\tau,\theta_\tau){1_{\{\tau< \infty\}}}\right]\\
		&=E_{x,t,i}\Bigg[\int_0^{\tau\wedge\nu_1} e^{-\rho_s}\Pi(X_s,i)ds+\int_{\tau\wedge\nu_1}^\tau e^{-\rho_s}\Pi(X_s,\theta_s)ds\\
		&~~~~~~~~~~~~~~~~~~~~~~~~~~~~~~~~~~~~~~~~~~-{ e^{-\rho_\tau}}\left(h(X_\tau,i)1_{\{\tau<\nu_1\}}+h(X_\tau,\theta_\tau)1_{\{\nu_1\leq\tau<\infty\}}\right)\Bigg]\\
		&= { \int_{0}^{\infty} \lambda_i(u)e^{-\int_0^{u+t}\lambda_i(s)ds}} E_{x,i}\Bigg[\Bigg(\int_0^{\tau\wedge u} e^{-\rho_s}\Pi(X_s,i)ds+\int_{\tau\wedge u}^\tau e^{-\rho_s}\Pi(X_s,\theta_s)ds\\
		&~~~~~~~~~~~~~~~~~~~~~~~~~~~~~~~~~~~~~~~~~-{ e^{-\rho_\tau}} \left(h(X_\tau,i)1_{\{\tau<u\}}+h(X_\tau,\theta_\tau)1_{\{u\leq\tau<\infty\}}\right)\Bigg)\Big\vert\nu_1=u\Bigg] {du},
	\end{align*}
	where the last equality follows in light of Equation \eqref{probability nu1} and Fubini's Theorem.

	Let $U_N\subset I\times [0,\infty)$ be a compact set, such that $U_N\nearrow \overline I\times [0,\infty)$, and fix $\omega\in\Omega$. Due to the continuity of $(s,x)\to X_s^x(\omega)$, the functions $(s,x)\to r(X^x_s(\omega),i)$, $(s,x)\to \Pi(X^x_s(\omega),i)$ and $(s,x)\to h(X^x_s(\omega),i)$ have maximum and minimum {on the set} $U_N$, namely
	\begin{align*}
		r(X^x_s(\omega),i)&\in[\epsilon_i,\tilde{r}_N(\omega,i)],\\
		\Pi(X^x_s(\omega),i)&\in[\underaccent{\tilde}{{\Pi}}_N(\omega,i),\tilde{\Pi}_N(\omega,i)],\\
		h(X^x_s(\omega),i)&\in[\underaccent{\tilde}{{h}}_N(\omega,i),\tilde{h}_N(\omega,i)].
	\end{align*}  
	Let $(x',t',i)\in I\times[0,\infty)\times\Theta$, then, for a fixed $\tau$, it follows from the Dominated Convergence Theorem that 
	\begin{align*}
		&\lim_{(x,t)\to(x',t')}J(x,t,i,\tau)= {\int_{0}^{\infty}\lambda_i(u)e^{-\int_0^{u+t'}\lambda_i(s)ds}}
		E_{x',i}\Bigg[\Bigg(\int_0^{\tau\wedge u} e^{-\rho_s}\Pi(X_s(\omega),i)ds\\
		&~~~~~~~~~~~~~~~~~~~~~+\int_{\tau\wedge u}^\tau e^{-\rho_s}\Pi(X_s(\omega),\theta_s)ds-e^{-\rho_\tau}\Big(h(X_\tau(\omega),i)1_{\{\tau<u\}}\\
		&~~~~~~~~~~~~~~~~~~~~~+h(X_\tau(\omega),\theta_\tau)1_{\{u\leq\tau<\infty\}}\Big)\Bigg)\Big\vert\nu_1=u\Bigg]{ du}=J(x',t',i,\tau^{x',t',i}_N).
	\end{align*}
	Let $\tau^{x',t',i}$ and $\tau_{U_N}$ be ${\cal F}_s-$stopping times, assuming that $(X_0,\zeta_0,\theta_0)=(x',t',i)$. Then, if $\tau_{U_N}=\inf\{s\geq 0:(s,X_s)\notin U_N\}$, and $\tau^{x',t',i}_N=\tau^{x',t',i}\wedge\tau_{U_N}$,
	The result holds true if
	\begin{align}\label{continuity of J}
		\lim_{U_N\nearrow \overline I\times [0,\infty)}J\left(x',t',i,\tau^{x',t',i}_N\right)=J\left(x',t',i,\tau^{x',t',i}\right).
	\end{align}

	To prove Equation \eqref{continuity of J}, we fix $\tau\in{\cal S}$ and we notice that, as $U_N\nearrow I\times [0,\infty)$,
	\begin{align*}
		0\leq\int_{t}^{\tau\wedge\tau_{U_N}}e^{-\rho_s}\Pi^\pm(X_s,\theta_s)ds\nearrow \int_{t}^{\tau\wedge T^I}e^{-\rho_s}\Pi^\pm(X_s,\theta_s)ds.
	\end{align*}
	Then, by utilizing the Monotone Convergence Theorem, it follows that
	\begin{align}\label{Auxiliary limit 1}
		&\lim_{U_N\nearrow \overline I\times [0,\infty)}E_{x',t',i}\left[\int_{0}^{\tau\wedge\tau_{U_N}}e^{-\rho_s}\Pi(X_s,\theta_s)ds\right]=E_{{x',t'},i}\left[\int_{0}^{\tau\wedge T^I}e^{-\rho_s}\Pi(X_s,\theta_s)ds\right].
	\end{align}
	Furthermore, since $\{h(X_\tau,\theta_\tau)\}_{\tau\in{\cal S}}$ is a uniformly integrable family of random variables and $e^{-\rho_{\tau\wedge\tau_{U_N}}}h(X_{\tau\wedge\tau_{U_N}},\theta_{\tau\wedge\tau_{U_N}})\to e^{-\rho_{\tau\wedge T^I}}h(X_{\tau\wedge T^I},\theta_{\tau\wedge T^I})$, $P-$almost surely, then,
	\begin{align}\label{Auxiliary limit 2}
		\lim_{U_N\nearrow \overline I\times [0,\infty)}E_{x',t',i}\left[e^{-\rho_{\tau\wedge\tau_{U_N}}}h(X_{\tau\wedge\tau_{U_N}},\theta_{\tau\wedge\tau_{U_N}})\right]=E_{x',t',i}\left[e^{-\rho_{\tau\wedge T^I}}h(X_{\tau\wedge T^I},\theta_{\tau\wedge T^I})\right].
	\end{align}
	Since $\tau\in{\cal S}$ is arbitrary, Equation \eqref{continuity of J} holds true and, thus, we finish the proof. 
\end{proof}	
To prove the DPP, we will introduce the following concept:
\begin{defi}\label{epsilon optimal strategy}
	Given the initial condition $(X_0,\zeta_0,\theta_0)=(x,t,i)$, the ${{\cal F}_s}-$stopping time $\tau^{x,t,i}_\epsilon$ is an $\epsilon-$optimal strategy if
	$$
	J(x,t,i,\tau^{x,t,i}_\epsilon)\geq V^*(x,t,i)-\epsilon,\text{ for some }\epsilon\geq 0.
	$$
\end{defi}
We note that, for each $(x,t,i)\in I\times[0,\infty)\times\Theta$, {an $\epsilon-$optimal} strategy always exists in light of the definition of the value function and Assumption \ref{Assumption : last one}. 

Henceforward, we denote the lower and upper semicontinuous envelopes of a locally bounded function $\phi:I\times[0,\infty)\times\Theta\to\mathbb{R}$, with respect to the variables $x$ and $t$, by:
\begin{align*}
	\underline{\phi}(x,t,i)\equiv\liminf_{(y,s)\to(x,t)}\phi(y,s,i),\\
	\overline{\phi}(x,t,i)\equiv\limsup_{(y,s)\to(x,t)}\phi(y,s,i).
\end{align*}
\begin{prop}\label{DPP}
	Let $(x,t,i)\in I\times[0,\infty)\times\Theta$ and $\delta\in{\cal S}$ be such that $\delta<\infty$. Then
	$$
	{V}^*(x,t,i)\leq \sup_{\tau\in{\cal S}}E_{x,t,i}\left[\int_{{ 0}}^{\tau\wedge\delta} e^{-\rho_s}\Pi({X}_s,\theta_s)ds+e^{-\rho_{\tau\wedge\delta}}\left(h(X_\tau,\theta_\tau)1_{\{\tau<\delta\}}+\overline{V}^*(X_\delta,\zeta_\delta,{\theta_\delta})1_{\{\tau\geq\delta\}}\right)\right],
	$$
	and
	$$
	{V}^*(x,t,i)\geq \sup_{\tau\in{\cal S}}E_{x,t,i}\left[\int_{{ 0}}^{\tau\wedge\delta} e^{-\rho_s}\Pi({X}_s,\theta_s)ds+e^{-\rho_{\tau\wedge\delta}}\left(h(X_\tau,\theta_\tau)1_{\{\tau<\delta\}}+\underline{V}^*(X_\delta,\zeta_\delta,{\theta_\delta})1_{\{\tau\geq\delta\}}\right)\right].
	$$
\end{prop} 
\begin{proof}
	The first inequality can be easily obtained, since for any $\delta\in{\cal S}$, such that $\delta<\infty$
	\begin{align*}
		\int_0^{{\tau}} e^{-\rho_s}\Pi({X}_s,\theta_s)ds&-e^{-\rho_{{\tau}}}h(X_{{\tau}},\theta_{{\tau}})1_{\{\tau<\infty\}}
		= \int_0^{\tau\wedge\delta} e^{-\rho_s}\Pi({X}_s,\theta_s)ds-e^{-\rho_{\tau}}h(X_{{\tau}},\theta_{{\tau}})
		1_{\{\tau<\delta\}}\\
		&~+e^{-\rho_\delta}\left(\int_{\delta}^{{\tau}} e^{-(\rho_{s}-\rho_\delta)}\Pi({X}_s,\theta_s)ds-e^{-(\rho_{\tau}-\rho_{\delta})}h(X_{\tau},\theta_{\tau})
		1_{\{\tau<\infty\}}\right)1_{\{\delta\leq\tau\}}.
	\end{align*}
	Due to the strong Markov property, it follows that
	\begin{align*}
		J(x,t,i,\tau)&\leq E_{x,t,i}\left[\int_{{0}}^{\tau\wedge\delta} e^{-\rho_s}\Pi({X}_s,\theta_s)ds+e^{-\rho_{\tau\wedge\delta}}\left(h(X_\tau,\theta_\tau)1_{\{\tau<\delta\}}+J(X_\delta,\zeta_\delta,{\theta_\delta},\tau)1_{\{\tau\geq\delta\}}\right)\right].
	\end{align*}
	Since $J(X_\delta,\zeta_\delta,{\theta_\delta},\tau)\leq V^*(X_\delta,\zeta_\delta,{\theta_\delta})\leq \overline V^*(X_\delta,\zeta_\delta,{\theta_\delta})$, the result follows by applying the supremum over $\tau\in\cal{S}$ to the previous inequality.
	
	To prove the second inequality, we fix $i\in\Theta$. Additionally, we note that $V^*(x,t,i)\geq-h(x,i)$, for every $(x,t)\in I\times[0,\infty)$, which is also true for $\underline V^*$ $\left(\underline V^*(x,t,i)\geq-h(x,i)\right)$, due to the continuity of the function $h(\cdot,i)$. Consequently, from Assumption \ref{Assumption : last one}, there is a bounded continuous function $\phi(\cdot,\cdot,i):I\times[0,\infty)\to\mathbb{R}$ such that $\underline V^*(x,t,i)\geq\phi(x,t,i)$.
	
	Fix $(x,t)\in I\times[0,\infty)$ and let $\tau^{x,t,i}_{\epsilon}$ be an $\epsilon-$optimal strategy, for some $\epsilon>0$, as defined in Definition \ref{epsilon optimal strategy}. Taking into account Lemma \ref{L continuity of J} and the continuity of $\phi(\cdot,\cdot,i)$, there is a sequence $\{\gamma_{(x,t,i)}\}_{(x,t)\in I\times[0,\infty)}\subset]0,\infty[$, such that, for every $(x',t')\in B_{\gamma_{(x,t,i)}}(x,t)$,
	\begin{align*}
		J(x',t',i,\tau^{x,t,i}_{\epsilon})-J(x,t,i,\tau^{x,t,i}_{\epsilon})>-\epsilon\quad\text{and}\quad \phi(x',t',i)-\phi(x,t,i)<\epsilon.
	\end{align*} 
	Naturally, $\{B_{\gamma_{(x,t,i)}}(x,t):(x,t)\in I\times[0,\infty)\}$ is an open cover of $I\times[0,\infty)$, and, therefore, in light of the Lindel\"of's Covering Theorem, there is a sequence $\{(x_j,t_j)\}_{j\in\mathbb{N}}\subset I\times[0,\infty)$, such that $\{B_{\gamma_{(x_j,t_j,i)}}(x_j,t_j)\}_{j\in\mathbb{N}}$ forms an open subcover of $I\times[0,\infty)$. Therefore, for $(x,t)\in B_{\gamma_{(x_j,t_j,i)}}(x_j,t_j)$, with $j\in\mathbb{N}$, we have
	\begin{align*}
		J(x,t,i,\tau^{x,t,i}_{\epsilon})>J(x_j,t_j,i,\tau^{x_j,t_j,i}_{\epsilon})-\epsilon\geq V^*(x_j,t_j,i)-2\epsilon\geq \phi(x_j,t_j,i)-2\epsilon>\phi(x,t,i)-3\epsilon.
	\end{align*}
	
	From the previous arguments, it is clear that  $\{A_j\}_{j\in\mathbb{N}}$, with $$A_j=B_{\gamma_{(x_j,t_j,i)}}(x_j,t_j)\setminus\bigcup_{n=1}^{j-1}B_{\gamma_{(x_n,t_n,i)}}(x_n,t_n),$$
	is also a finite subcover of $I\times[0,\infty)$, verifying $A_j\cap A_n=\emptyset$, with $j\neq n$. Consider $\delta,\tau\in{\cal S}$ with $\delta<\infty$, then one can build the strategy
	\begin{align*}
		\overline\tau=\tau 1_{\{\tau<\delta\}}+\tau_\epsilon 1_{\{\tau\geq\delta\}},\text{ with }\tau_\epsilon=\sum_{i=1}^k\left(\sum_{j=1}^n\tau_\epsilon^{x_j,t_j,i}1_{\{(X_\delta,\zeta_\delta)\in A_j\}}\right)1_{\{\theta_\delta=i\}}+\delta1_{\{(X_\delta,\zeta_\delta)\notin \cup_{j=1}^n A_j\}},
	\end{align*}
	which trivially belongs to $\cal S$. Taking into account the decomposition
	\begin{align*}
		\int_0^{{\overline\tau}} e^{-\rho_s}\Pi({X}_s,\theta_s)ds&-e^{-\rho_{{\tau_\epsilon}}}h(X_{{\overline\tau}},\theta_{{\overline\tau}})1_{\{\overline\tau<\infty\}}
		= \int_0^{\tau\wedge\delta} e^{-\rho_s}\Pi({X}_s,\theta_s)ds-e^{-\rho_{\tau}}h(X_{{\tau}},\theta_{{\tau}})
		1_{\{\tau<\delta\}}\\
		&~+e^{-\rho_\delta}\left(\int_{\delta}^{{\tau_\epsilon}}e^{-(\rho_{s}-\rho_\delta)}\Pi({X}_s,\theta_s)ds-e^{-(\rho_{\tau_\epsilon}-\rho_{\delta})}h(X_{\tau_\epsilon},\theta_{\tau_\epsilon})1_{\{\tau_\epsilon<\infty\}}
		\right)1_{\{\delta\leq\tau\}},
	\end{align*}
	we have
	\begin{align*}
		&V^*(x,t,i)\geq J(x,t,i,\overline\tau)\geq E_{x,t,i}\Bigg[ \int_0^{\tau\wedge\delta} e^{-\rho_s}\Pi({X}_s,\theta_s)ds-e^{-\rho_{\tau}}h(X_{{\tau}},\theta_{{\tau}})
		1_{\{\tau<\delta\}}\Bigg]+\\
		&~~~~~~~+E_{x,t,i}\Bigg[\left(e^{-\rho_\delta}\left(\phi(X_\delta,\zeta_\delta,\theta_\delta)-3\epsilon\right)1_{(X_\delta,\zeta_\delta)\in \cup_{j=1}^n A_j}-e^{-\rho_\delta}h(X_\delta,\theta_\delta)1_{(X_\delta,\zeta_\delta)\notin \cup_{j=1}^n A_j}\right)1_{\{\delta\leq\tau\}}\Bigg].
	\end{align*}
	Since $\phi(\cdot,\cdot,i)$ is a bounded continuous function, for every $i\in\Theta$, and $\{h(X_\delta,\theta_\delta)\}_{\{\delta\in{\cal S}\}}$ is a uniformly integrable family of random variables, from the Dominated Convergence Theorem we get
	\begin{align*}
		&V^*(x,t,i)\geq  E_{x,t,i}\Bigg[ \int_0^{\tau\wedge\delta} e^{-\rho_s}\Pi({X}_s,\theta_s)ds-e^{-\rho_{\tau}}h(X_{{\tau}},\theta_{{\tau}})
		1_{\{\tau<\delta\}}+\\
		&~~~~~~~~~~~~~~~~~~~~~~~~~~~~~~~~~~~~~~~~~~~~~~~~~~~~~~~~~~~~~~~~~~~~~~~~+e^{-\rho_\delta}\left(\phi(X_\delta,\zeta_\delta,\theta_\delta)-3\epsilon\right)1_{\{\delta\leq\tau\}}\Bigg]\\
		&~~~~~~\geq  E_{x,t,i}\Bigg[ \int_0^{\tau\wedge\delta} e^{-\rho_s}\Pi({X}_s,\theta_s)ds-e^{-\rho_{\tau}}h(X_{{\tau}},\theta_{{\tau}})
		1_{\{\tau<\delta\}}+e^{-\rho_\delta}\phi(X_\delta,\zeta_\delta,\theta_\delta)1_{\{\delta\leq\tau\}}\Bigg]-3\epsilon.
	\end{align*}
	Now, we pick a monotonically increasing sequence of bounded continuous functions $\{\phi_n\}_{n\in\mathbb{N}}$, such that $\phi_n(x,t,i)\to\underline V^*(x,t,i)$ as $n\to\infty$ (this sequence exists in light of Urysohn's Lemma) and, consequently, from the Monotones Convergence Theorem we get 
	$$
	\lim_{n\to\infty}E_{x,t,i}\Bigg[e^{-\rho_\delta}\phi_n(X_\delta,\zeta_\delta,\theta_\delta)1_{\{\delta\leq\tau\}}\Bigg]=E_{x,t,i}\Bigg[e^{-\rho_\delta}\underline V^*(X_\delta,\zeta_\delta,\theta_\delta)1_{\{\delta\leq\tau\}}\Bigg].
	$$
	As $\epsilon$ is arbitrary, we obtain the result. 
\end{proof} 

\subsection{Viscosity solutions}
Assuming that $V^*$ is sufficiently regular, it is not difficult to show that $V^*$, in the classical sense, satisfies the system of HJB equations
\begin{align}\begin{cases}
		\label{Eq HJB}
		&F_i(x,t,\{v(x,t,j):\,j\in\Theta\},\partial_tv(x,t,i),Dv(x,t,i),D^2v(x,t,i))=0\\
		&(x,t,i)\in I\times{\Theta}\times(0,\infty)
	\end{cases},
\end{align}
where 
\begin{align*}
	&F_i(x,t,\{v(x,t,j):\,j\in\Theta\},\partial_tv(x,t,i),Dv(x,t,i),D^2v(x,t,i))\equiv\\ &~~~~~~~~~~~~~~~~~~~~~~~~~~~~~~~~~~~~~~~~~~~~~~~~~\equiv\min\Big\{-(\tilde{\cal L}v)(x,t,i)-\Pi(x,i),\,v(x,t,i)+h(x,i)\Big\},
\end{align*}
and $\partial_tv$, $Dv$ and $D^2v$ are, respectively, the first derivative of $v$ in $t$, the vector of first derivatives of $v$ in $x$ and the matrix of second derivatives of $v$ in $x$. Furthermore, the following boundary condition must also be satisfied
\begin{align}
	\label{boundary condition}
	v(x,t,i)=-h(x,i),\quad \text{for all } x\in \partial I .
\end{align}
One can note that this boundary condition is trivially satisfied when $I=D$ since, in this case, $\partial I=\emptyset$. 
Throughout this section, we will prove that the value function, $V^*$, is a viscosity solution to the system of coupled HJB equations \eqref{Eq HJB} and the boundary condition \eqref{boundary condition}.

Before we state the main result of this section, we introduce the definition of viscosity solutions for systems of variational inequalities, following the work of Ishii and Koike \cite{ishii1991viscosity1} (see also Crandall, Ishii and Lions \cite{crandall1992user}). 

\begin{defi}\label{Definition Viscosity Solution}
	Consider a locally bounded function $v:I\times(0,\infty)\times\Theta\to\mathbb{R}$. Then, $v$ is a
	\begin{itemize}
		\item[(a)] viscosity subsolution to \eqref{Eq HJB} if whenever $\psi\in C^2(I\times[0,\infty))$, $i\in\Theta$ and $\overline{v}(\cdot,\cdot,i)-\psi(\cdot,\cdot)$ has a local maximum at $(x,t)\in I\times [0,\infty)$, such that $\overline{v}(x,t,i)=\psi(x,t)$, then
		\begin{align*}\label{subsolution 1}
			&{F}_{i}(x,t,\overline{v}(x,t,i),D\psi(x,t),D^2\psi(x,t);\{\overline{v}(x,t,j):\,j\in\Theta,\,j\neq i\})\leq 0.
		\end{align*}
		\item[(b)] viscosity supersolution to \eqref{Eq HJB} if whenever $\psi\in C^2(I\times[0,\infty))$, $i\in\Theta$ and $\underline{v}(\cdot,\cdot,i)-\psi(\cdot,\cdot)$ has a local minimum at $(x,t)\in I\times[0,\infty)$, such that $\underline{v}(x,t,i)=\psi(x,t)$, then
		\begin{align*}
			&{F}_{i}(x,t,\underline{v}(x,t,i),D\psi(x,t),D^2\psi(x,t);\{\underline{v}(x,t,j):\,j\in\Theta,\,j\neq i\})\geq 0.
		\end{align*}
		\item[(c)] viscosity solution to \eqref{Eq HJB} if it is simultaneously a viscosity subsolution and a viscosity supersolution to \eqref{Eq HJB}. 
	\end{itemize}
\end{defi}
We note that in Definition \ref{Definition Viscosity Solution}, without loss of generality, we can consider functions $\psi\in C^2_0(I\times[0,\infty))$, instead of $\psi\in C^2(I\times[0,\infty))$. Additionally, in the previous definition the term ``local" can be replaced by either ``strict local" or ``global".
\begin{prop}\label{V is a viscosity solution}
	Let ${V}^*$ be the value function defined as in \eqref{Optimal Stopping Problem}. Then $V^*$ is a viscosity solution to the system of equations \eqref{Eq HJB} and the boundary condition \eqref{boundary condition} is satisfied.
\end{prop}
\begin{proof} In light of to Assumption \ref{Assumption : last one}, the function $V^*$ is locally bounded. Therefore, we will proceed through the viscosity supersolution and subsolution properties and the boundary condition.
	
	\textbf{Supersolution property:} To see that ${V}^*$ is a viscosity supersolution to the system of equations \eqref{Eq HJB}, we fix $i\in\Theta$, and let $(\overline{x},\overline{t})\in I\times[0,\infty)$ and $\psi\in C^2_0(I\times[0,\infty))$ be such that $(\overline{x},\overline{t})$ is a local minimizer of $\underline{V}^*(\cdot,\cdot,i)-\psi(\cdot,\cdot)$ and $\underline{V}^*(\overline{x},\overline{t},i)-\psi(\overline{x},\overline{t})= 0$.
	
	Fix a sufficiently small $\epsilon>0$ and let $B_\epsilon(\overline{x},\overline{t})$ be a ball centered in $(\overline{x},\overline{t})$ with radius $\epsilon$. Furthermore, let $\{(x_n,t_n)\}_{n\in\mathbb{N}}\subset B_\epsilon(\overline{x},\overline{t})$ be such that 
	$$
	\left(x_n,t_n,\underline V^*(x_n,t_n,i)\right)\to\left(\overline{x},\overline{t},{V}^*(\overline{x},\overline{t},i)\right),
	$$
	as $n\to\infty$. Naturally, such a sequence exists in light of the definition of $\underline{V}^*$. Throughout the proof, we are interested in the process $(X_s^n,\zeta^n_s,\theta_s)$ which represents $(X_s,\zeta_s,\theta_s)$ when $(X_0,\zeta_0,\theta_0)=(x_n,t_n,i)$. Whenever there is no risk of misunderstanding, we will simple write $(X_s,\zeta_s,\theta_s)$. 
	
	If $\{\eta_n\}_{n\in\mathbb{N}}$ is such that $\eta_n\to 0$ and $$\tau_{n,\epsilon}\equiv\inf\{s>0\,:\vert X_s^n-x_n\vert\geq\epsilon,\,\zeta^n_s\geq t_n+\sqrt{\eta_n}\}\wedge\inf\{s>0\,:\theta_s-\theta_{s^-}\neq 0\},$$
	for some $\epsilon>0$, then, it follows from the DPP that 
	\begin{align*}
		0\geq  E_{x_n,t_n,i}\left[\int_{0}^{\tau_{n,\epsilon}} e^{-\rho_s}\Pi({X}_s,\theta_s)ds+e^{-\rho_{\tau_{n,\epsilon}}}\underline{V}^*({X}_{\tau_{n,\epsilon}},\zeta_{\tau_{n,\epsilon}},\theta_{\tau_{n,\epsilon}})\right]-V^*(x_n.t_n,i).
	\end{align*}
	Now, consider the auxiliary function $\Psi$ given by
	\begin{equation*}
		\Psi(x,t,j)=\begin{cases}
			\psi(x,t),& \text{if }j=i\\
			\underline{V}^*(\overline x,\overline t,j),& \text{if }j\neq i
		\end{cases}.
	\end{equation*}
	From the Dynkin's formula we get that
	\begin{align*}\nonumber
		E_{x_n,t_n,i}[e^{-\rho_{\tau_{n,\epsilon}}}\Psi(X_{\tau_{n,\epsilon}},\zeta_{\tau_{n,\epsilon}},\theta_{\tau_{n,\epsilon}})]&=E_{x_n,t_n,i}[e^{-\rho_{\tau_{n,\epsilon}}}\Psi(X_{\tau_{n,\epsilon}},\zeta_{\tau_{n,\epsilon}},i)]\\
		&=\psi(x_n,t_n)+E_{x_n,t_n,i}\left[\int_{0}^{\tau_{n,\epsilon}} e^{-\rho_{s}}(\tilde{\cal L}\Psi)(X_s,\zeta_s,i)ds\right],
	\end{align*}
	where
	\begin{align*}
		(\tilde{\cal L}\Psi)(\overline{x},\overline{t},i)&=\frac{\partial\phi}{\partial t}(\overline{x},\overline{t})-r(\overline{x},i)\underline{V}^*(\overline{x},\overline{t},i)+\alpha(\overline{x},i)\cdot D{\phi}(\overline{x},\overline{t})+\frac{1}{2}Tr\left[\sigma\sigma^T(\overline{x},i)D^2{\phi}(\overline{x},\overline{t})\right]\\ &+\sum_{j\neq i}\lambda_{i,j}(\overline{t})\left(\underline{V}^*(\overline{x},\overline{t},j)-\underline{V}^*(\overline{x},\overline{t},i)\right).
	\end{align*}
	Since there are $\epsilon>0$ such that $\underline{V}^*({x},{t},i)\geq \psi(x,t)=\Psi(x,t,i)$, in $B_\epsilon(\overline{x},\overline{t})$, we can choose $n,\epsilon$ such that
	\begin{align*}
		V^*(x_n,t_n,i)&\geq  E_{x_n,t_n,i}\left[\int_0^{\tau_{n,\epsilon}} e^{-\rho_s}\Pi({X}_s,\theta_s)ds+e^{-\rho_{\tilde\tau_{n,\epsilon}}}\Psi({X}_{\tau_{n,\epsilon}},\zeta_{\tau_{n,\epsilon}},\theta_{\tau_{n,\epsilon}})\right].
	\end{align*} 
	Therefore, fixing $\eta_n\equiv V^*(x_n,t_n,i)-\Psi(x_n,t_n,i)\to 0$ as $n\to\infty$, we obtain
	\begin{align*}
		\eta_n&\geq  E_{x_n,t_n,i}\left[\int_0^{\tau_{n,\epsilon,h}} e^{-\rho_s}\Pi({X}_s,\theta_s)ds+e^{-\rho_{\tilde\tau_{n,\epsilon}}}\Psi({X}_{\tau_{n,\epsilon,h}},\zeta_{\tau_{n,\epsilon,h}},\theta_{\tau_{n,\epsilon,h}})\right]-\Psi(x_n,t_n,i)\\
		&= E_{x_n,t_n,i}\left[\int_0^{\tau_{n,\epsilon,h}} e^{-\rho_s}\left(\Pi({X}_s,i)+(\tilde{\cal L}\Psi)(X_s,\zeta_s,i)\right)ds\right].
	\end{align*} 
	Assuming, without loss of generality, that $\eta_n\to 0$ but $\eta_n\neq 0$, we have 
	\begin{align*}
		\sqrt{\eta_n}&\geq E_{x_n,t_n,i}\left[\frac{1}{\sqrt{\eta_n}}\int_0^{\tau_{n,\epsilon}} e^{-\rho_s}\left(\Pi({X}_s,i)+(\tilde{\cal L}\Psi)(X_s,\zeta_s,i)\right)ds\right].
	\end{align*}
	By letting $n\to \infty$,  we obtain
	\begin{align*}
		0\geq -(\tilde{\cal L}v)(\overline x,\overline t,i)-\Pi(\overline x,i),\quad\text{for }i\in\Theta.
	\end{align*}
	To finish this part of the proof, we note that $V^*(\overline x,\overline t,i)\geq J(\overline x,\overline t,i,0)= -h(\overline x,i)\Rightarrow \underline{V}^*(\overline x,\overline t,i)\geq -h(\overline x,i)$, because $h$ is continuous.
	
	\textbf{Subsolution property:} To prove that $V^*$ is a viscosity subsolution to the system of HJB equations \eqref{Eq HJB}, we argue by contradiction.
	
	Fix $i\in\Theta$, and let $(\overline{x},\overline{t})\in I\times[0,\infty)$ and $\psi\in C^2_0(I\times[0,\infty))$ be such that $(\overline{x},\overline{t})$ is a {strict maximizer} of $\overline{V}^*(\cdot,\cdot,i)-\psi(\cdot,\cdot)$ and $\overline{V}^*(\overline{x},\overline{t},i)-\psi(\overline{x},\overline{t})= 0$.  To get a contradiction, we also assume that
	\begin{equation}
		F_i(\overline{x},\overline{t},\{\overline{V}^*(\overline{x},\overline{t},j):\,j\in\Theta\},\partial_t\psi(\overline{x},\overline{t},i),D\psi(\overline{x},\overline{t},i),D^2\psi(\overline{x},\overline{t},i))>0.
	\end{equation}
	Taking into account that $h(\cdot,i)$ is continuous, for all $i\in{\Theta}$, there is $\epsilon>0$ such that 
	\begin{equation}\label{contradiction 1}
		\psi(\overline{x},\overline{t})+h(\overline{x},i)\geq\epsilon\quad \text{and} \quad -(\tilde{\cal L}\psi)(\overline{x},\overline{t},i)-\Pi(\overline{x},i)\geq 0,\quad\text{for all }(\overline{x},\overline{t})\in {\cal V}_\epsilon(\overline{x},\overline{t}), 
	\end{equation}
	where, ${\cal V}_{\epsilon}(\overline{x},\overline{t})$ is a neighborhood of $(\overline{x},\overline{t})$ of the form ${\cal V}_{\epsilon}(\overline{x},\overline{t})=B_{\epsilon}(\overline{x})\times[\overline{t},\overline{t}+\epsilon[$, and $B_{\epsilon}$ is a ball centered in $\overline{x}$ with radius $\epsilon$.  Additionally, since $(\overline{x},\overline{t})$ is a strict maximizer, there is $\delta<0$ such that 
	\begin{equation}\label{contradiction 2}
		\max_{(x,t)\in\partial {\cal V}_\epsilon(\overline{x},\overline{t})}\left(\overline{V}^*(x,t,i)-\psi(x,t)\right)=\delta.
	\end{equation}
	In light of the definition of $\overline{V}^*$, there is $\{(x_n,t_n)\}_{n\in\mathbb{N}}\subset I\times[0,\infty)$, such that
	$$
	\left(x_n,t_n,\overline V^*(x_n,t_n,i)\right)\to\left(\overline{x},\overline{t},{V}^*(\overline{x},\overline{t},i)\right).
	$$  
	Now, we define the stopping time $\tau_{n,\epsilon}\equiv\inf\{s\geq 0\,: (X_s^n,\zeta_s^n)\notin{\cal V}_\epsilon(\overline{x},\overline{t})\}$ and the function 
	$$
	\Psi(x,t,j)=\begin{cases}
	\psi(x,t),&j= i\\
	\overline V^*(\overline x,\overline t,j),&j\neq i
	\end{cases}.
	$$ 
	If $\eta_n=V^*(\overline{x},\overline{t},i)-\psi(\overline{x},\overline{t})$, from the Dynkin's formula, it follows that
	\begin{align}\label{contradiction 3}
		V^*(\overline{x},\overline{t},i)&=\eta_n+\psi(\overline{x},\overline{t})=\eta_n+ \Psi(\overline{x},\overline{t},i)\\\nonumber
		& =\eta_n+ E\left[e^{-\rho_{\tau\wedge\tau_{n,\epsilon}}}\Psi(X_{\tau\wedge\tau_{n,\epsilon}},\zeta_{\tau\wedge\tau_{n,\epsilon}},\theta_{\tau\wedge\tau_{n,\epsilon}})-\int_0^{\tau\wedge\tau_{n,\epsilon}}e^{-\rho_s}(\tilde{\cal L}\Psi)(X_s,\zeta_s,\theta_s)ds\right]\\\nonumber
		&\geq \eta_n+ E\left[e^{-\rho_{\tau\wedge\tau_{n,\epsilon}}}\Psi(X_{\tau\wedge\tau_{n,\epsilon}},\zeta_{\tau\wedge\tau_{n,\epsilon}},\theta_{\tau\wedge\tau_{n,\epsilon}})+\int_0^{\tau\wedge\tau_{n,\epsilon}}e^{-\rho_s}\Pi(X_s,\theta_s)ds\right].
	\end{align}  
	The inequality follows in light of the right-hand side of \eqref{contradiction 1}. 
	
	Choosing $\epsilon>0$ such that $\tau_{n,\epsilon}<\inf\{s\geq 0\,:\theta_s\neq i\}$, from the left-hand side of \eqref{contradiction 1}, we get that 
	\begin{align}\label{contradiction 4}
		&E\left[e^{-\rho_{\tau\wedge\tau_{n,\epsilon}}}\Psi(X_{\tau\wedge\tau_{n,\epsilon}},\zeta_{\tau\wedge\tau_{n,\epsilon}},\theta_{\tau\wedge\tau_{n,\epsilon}})\right]=\\\nonumber
		&=E\left[e^{-\rho_{\tau}}\psi(X_{\tau},\zeta_{\tau})1_{\{\tau<\tau_{n,\epsilon}\}}\right]+E\left[e^{-\rho_{\tau_{n,\epsilon}}}\psi(X_{\tau_{n,\epsilon}},\zeta_{\tau_{n,\epsilon}})1_{\{\tau\geq\tau_{n,\epsilon}\}}\right]\\\nonumber
		&=E\left[e^{-\rho_{\tau}}(-h(X_{\tau},\theta_{\tau})+\epsilon)1_{\{\tau<\tau_{n,\epsilon}\}}+e^{-\rho_{\tau_{n,\epsilon}}}\left(\overline V^*(X_{\tau_{n,\epsilon}},\zeta_{\tau_{n,\epsilon}},i)-\delta\right)1_{\{\tau\geq\tau_{n,\epsilon}\}}\right]\\\nonumber
		&\geq E\left[-e^{-\rho_{\tau}}h(X_{\tau},\theta_{\tau})1_{\{\tau<\tau_{n,\epsilon}\}}+e^{-\rho_{\tau_{n,\epsilon}}}\overline V^*(X_{\tau_{n,\epsilon}},\zeta_{\tau_{n,\epsilon}},i)1_{\{\tau\geq\tau_{n,\epsilon}\}}\right]+\min(\epsilon,-\delta)E\left[e^{-\rho_{\tau\wedge\tau_{n,\epsilon}}}\right].
	\end{align}
	Since $E\left[e^{-\rho_{\tau\wedge\tau_{n,\epsilon}}}\right]>0$, by combining the calculations made in \eqref{contradiction 3} and \eqref{contradiction 4}, and taking into account that $\eta_n\to 0$ as $n\to \infty$, we obtain the desired contradiction in the DPP.
	
	\textbf{Boundary condition:} To finalize the proof, we must prove that $V^*(x,t,i)=h(x,i)$, for all $x\in\partial I$. In fact, if $X_0=x\in\partial I$, then $T^I=0$ $P-$almost surely and, consequently, $V^*(x,t,i)=J(x,t,i,\tau^*\wedge 0)=h(x,i)$.
\end{proof}
Until the end of this section, we will introduce some useful auxiliary results to prove the uniqueness result in the next section. From now on, $\hat{h}:(x,t,i)\to D\times[0,\infty)\times\Theta$ is such that $\hat{h}(\cdot,\cdot,i)$ is a continuous function. 
\begin{lemma}\label{remark}
	Consider the modified optimal stopping problem 
	\begin{equation*}
		V_{\Upsilon}(x,t,i)=\sup_{\tau\in{\cal S}}E_{x,t,i}\left[\int_0^{\tau\wedge\tau_{{}_\Upsilon}}e^{-\rho_s}\Pi(X_s,\theta_s)ds-e^{-\rho_{\tau\wedge \tau_{{}_\Upsilon}}}\hat{h}(X_{\tau\wedge\tau_{{}_\Upsilon}},\zeta_{\tau\wedge\tau_{{}_\Upsilon}},\theta_{\tau\wedge\tau_{{}_\Upsilon}})1_{\{\tau\wedge \tau_{{}_\Upsilon}<\infty\}}\right],
	\end{equation*}
	where $\tau_{{}_\Upsilon}=\inf\{s\geq 0:\zeta_s\geq\Upsilon\}$ and $\Upsilon >0$ is a deterministic and  finite time. In this case, the value function $V_{\Upsilon}:(x,t,i)\to\overline{I}\times[0,\Upsilon]\times\Theta$ is a viscosity solution to
	\begin{align}\label{Auxiliary boundary problem}
		\begin{cases}
			&\min\Big\{-(\tilde{\cal L}v)(x,t,i)-\Pi(x,i),v(x,t,i)-\hat h(x,t,i)\Big\}=0\\
			&v(x,t,i)=\hat h(x,t,i),\quad \forall (x,t,i)\in (\partial I\times [0,\Upsilon[\cup I\times\{\Upsilon\})
		\end{cases}.
	\end{align}
\end{lemma}
This can be easily proven by using similar arguments to the ones used in Proposition \ref{V is a viscosity solution}.
For future reference, we note that any viscosity solution $v$ to \eqref{Auxiliary boundary problem} is such that, for $i\in\Theta$, $v(\cdot,\cdot,i)$ satisfies the boundary problem
\begin{align}\label{ODE expected value}
	-\tilde{\cal L}v(x,t,i)-\Pi(x,i)=0,&\quad \text{for all }(x,t)\in A^i_v\\\label{ODE terminal condition}
	v(x,t,i)+\hat h(x,t,i)=0,&\quad \text{for all }(x,t)\in (\overline{I}\times[0,\Upsilon])\setminus A^i_v,
\end{align}
with \begin{equation}\label{Ai}
	A^i_v=\{(x,t)\in {I}\times[0,\Upsilon[~:~v(x,t,i)>-\hat h(x,t,i)\}.
\end{equation}

\begin{lemma}\label{v > h}
	Let $v:I\times[0,\infty)\times\Theta\to\mathbb{R}$ be a viscosity solution to \eqref{Auxiliary boundary problem}. Then, $v$ is a viscosity solution to the boundary problem \eqref{ODE expected value}-\eqref{ODE terminal condition}-\eqref{Ai} and verifies $v(x,t,i)\geq-\hat h(x,t,i)$. 
\end{lemma}
\begin{proof}
	Firstly, we prove that any viscosity supersolution $v$ to \eqref{Auxiliary boundary problem} verifies  $v(x,t,i)\geq-\hat h(x,t,i)$, for all $(x,t,i)\in I\times[0,T]\times\Theta$.
	Let $B_\epsilon(x,t)$ be a ball centered in $(x,t)$ with radius $\epsilon>0$, such that $\overline B_\epsilon(x,t)\subset I\times[0,\Upsilon[$. Since the function $\underline{v}(\cdot,\cdot,i)$, with $i\in\Theta$, is lower semicontinuous, then there is $$(\overline x,\overline t)=\arg\min\{\underline v(x,t,i)\,:(x,t)\in\overline B_\epsilon(x,y)\}.$$
	Therefore, by choosing $\psi(x,t)= \underline v(\overline x, \overline t,i)$, for all $(x,t)\in I\times[0,\Upsilon[$, $(\overline x, \overline t)$ is a local minimizer of $\underline v(x,t,i)-\psi(x,t)$ and $\underline v(\overline x, \overline t,i)=\psi(\overline x, \overline t)\geq -\hat h(\overline x,\overline t, i)$. Letting $\epsilon$ go to $0$ allows us to get that $\underline v(x,t,i)\geq-\hat h(x,t,i)\Rightarrow v(x,t,i)\geq-\hat h(x,t,i)$. Furthermore, $v(x,t,i)=-\hat h(x,t,i)$, for all $(t,x)\in I\times[0,\Upsilon]\setminus A_v^i$ and $i\in\Theta$.
	
	To finish this proof, we note that, if $v(x,t,i)> -\hat h(x,t,i)$, then $\overline v(x,t,i)> -\hat h(x,t,i)$. Therefore, since $v$ is a viscosity solution to \eqref{Auxiliary boundary problem}, then $v$ is a viscosity solution to Equation \eqref{ODE expected value}. 
\end{proof}

To introduce the next result, we start by defining the expected value $H(x,t,i)$ in the following way:
\begin{align}\label{expected value H}
	\begin{cases}&H(x,t,i)\equiv E_{x,t,i}\left[\int_0^{\tau_A}e^{-\rho_s}\Pi(X_s,\theta_s)ds-e^{-\rho_{\tau_A}}\hat{h}(X_{\tau_A},\zeta_{\tau_A},\theta_{\tau_A})1_{\{\tau_A<\infty\}}\right]\\
		&\tau_A=\inf\{s\geq 0:(X_s,\zeta_s,\theta_s)\notin A\}, \quad \text{and}\\
		&A=\cup_{i\in\Theta}A_i\times \{i\} \quad \text{with}\quad A_i\subset I\times[0,\Upsilon]\text{ an open set}.
	\end{cases}
\end{align}
In the next lemma, we characterize the expected value $H$ as a solution to the  boundary problem \eqref{expected value H}.
\begin{lemma}\label{Corollary existence}
	Let
	$H:I\times[0,\infty)\times\Theta\to\mathbb{R}$ be the function defined as in \eqref{expected value H}. Then $H$ is a viscosity solution to the boundary problem \eqref{ODE expected value}-\eqref{ODE terminal condition}, replacing $A_v^i$ by $A_i$ as in \eqref{expected value H}.
\end{lemma}
\begin{proof}
	First of all, we note that, in light of Assumption \eqref{Assumption : last one}, $H$ is locally bounded. Fixing $i\in\Theta$, by construction,
	\begin{align*}
		H(x,t,i)&=-h(x,i),\quad \text{for all }(x,t)\notin A^i.
	\end{align*}
	If $(x,t)\in A^i$ and $\tau\in{\cal S}$ is such that $\tau<\tau_A$, then
	\begin{equation*}
		H(x,t,i)=E_{x,t,i}\left[\int_0^{\tau}e^{-\rho_s}\Pi(X_s,\theta_s)ds+e^{-\rho_\tau}H(X_{\tau},\zeta_\tau,\theta_{\tau})\right].
	\end{equation*}
	Therefore, to finish the proof, one needs to demonstrate that $H$ is a viscosity solution to  \eqref{ODE expected value}, which is straightforward in light of the arguments used to prove Proposition \ref{V is a viscosity solution}.
\end{proof}

\subsection{The uniqueness result}

At this generality, the uniqueness of solutions to the system of HJB equations cannot be guaranteed without further conditions (see Example \ref{Example}). Therefore, in this section, we present some additional conditions that guarantee the uniqueness of the viscosity solution to \eqref{Eq HJB}. Under these conditions, such a solution will be $V^*$.

The next example was inspired by Example 3.1 in {\O}ksendal and Reikvam \cite{oksendal1998stochastic}.
\begin{example}\label{Example}
	Consider the following system of HJB equations:
	\begin{align}\label{HJB0 example}
		&\min\left\{-\frac{1}{2}\sigma^2v''(x,0),v(x,0)-\frac{x^2}{1+x^2}\right\}=0,\\
		&\min\left\{-\frac{1}{2}\sigma^2v''(x,1)-\lambda(v(x,0)-v(x,1)),v(x,1)-1\right\}=0, \quad \text{for all } x\in \mathbb{R}.\label{HJB1 example}
	\end{align}
	It is straightforward to see that any constant function $v(x,i)=a$ with $a\geq 1$ is a classical solution (and, consequently, a viscosity solution) to \eqref{HJB0 example}-\eqref{HJB1 example}. 
\end{example}

Uniqueness of viscosity solutions is generally guaranteed through a suitable comparison principle, which, in our case, exists if conditions C.1 and C.2 are satisfied (see Ishii and Koike \cite{ishii1991viscosity1}).
Henceforward, $\mathbb{S}_n$ denotes the set of all symmetric matrices of dimension $n$.
\begin{itemize}
	\item[C.1]There is a number $b>0$ such that if $m=(m_1,\ldots,m_k)$ and $n=(n_1,\ldots,n_k)\in \mathbb{R}^{k}$, $\max\limits_{u\in\Theta}(m_u-n_u)>0$ and $(x,t,p,a)
	\in I\times[0,\infty)\times\mathbb{R}^{n}\times\mathbb{R}$, then there is a $j=j(m,n,x,t,p,a)\in\Theta$ such that
	$$
	(m_j-n_j)=\max\limits_{u\in\Theta}(m_u-n_u),
	$$
	and, for all $X\in\mathbb{S}_n$,
	\begin{align*}
		{F}_j(x,t,\{m_i:i\in\Theta\},a,p,X)&-{F}_j(x,t,\{n_i:i\in\Theta\},a,p,X)\geq b(m_j-n_j).
	\end{align*}
	\item[C.2]There is a continuous function $\omega:[0,\infty)\to [0,\infty)$ with $w(0)=0$ such that if $X,Y\in {\mathbb{S}}_n$, $b>1$ and 
	\begin{equation*}
		-3b
		\left( {\begin{array}{cc}
				Id & 0 \\
				0 & Id \\
		\end{array} } \right)\leq \left( {\begin{array}{cc}
				X & 0 \\
				0 & Y \\
		\end{array} } \right)\leq 3b
		\left( {\begin{array}{cc}
				Id & -Id \\
				-Id & Id \\
		\end{array} } \right)
	\end{equation*}
	then, for all $j\in\Theta$,  $(x,t)$, $(y,s)\in I\times[0,\infty)$, and $m\in\mathbb{R}^{k}$
	\begin{align*}
		{F}_j(y,s&,\{m_i:i\in\Theta\},b(t-s),b(x-y),-Y)-\\
		&~~~~~~~~~~~~~~~~~-{F}_j(x,t,\{m_i:i\in\Theta\},b(t-s),b(x-y),X)\leq\omega\left(a\vert x-y\vert^2+\frac{1}{a}\right).
	\end{align*}
\end{itemize}
In the next lemma, we prove that conditions C.1 and C.2 are satisfied.
\begin{lemma}\label{Lemma conditions unique solution}
	Consider the system of HJB equations given by \eqref{Eq HJB} and assume that Assumption \ref{A functions} holds true. 
	Then, conditions C.1 and C.2 are satisfied in any compact set $U\subset I\times[0,\infty)$.
\end{lemma}
\begin{proof}To prove that condition C.1 is verified, we start by introducing the following notation:
	\begin{align*}\nonumber
		G_j(x,t,\{m_i:i\in\Theta\},a,p,X)&=r(x,j)m_j-a-\alpha(x,j)\cdot p-\frac{1}{2}Tr[\sigma\sigma^T(x,j)X]\\
		&-\sum_{i\neq j}\lambda_{j,i}(t)\left(m_{i}-m_j\right) -\Pi(x,j)\quad\text{for } i\neq j\in\Theta.
	\end{align*}
	Now, we assume that there is $j$ such that $$0<(m_j-n_j)=\max\limits_{u\in\Theta}(m_u-n_u).$$ Therefore,
	\begin{align*}
		{G}_j(x,t,\{m_i:i\in\Theta\},a,p,X)&={G}_j(x,t,\{n_i:i\in\Theta\},a,p,X)+\sum_{i\neq j}\lambda_{j,i}(t)\left(m_{j}-n_j\right)\\
		&-\sum_{i\neq j}\lambda_{j,i}(t)\left(m_{i}-n_i\right)+r(x,j)(m_j-n_j)\\
		&\geq G_j(x,t,\{n_i:i\in\Theta\},a,p,X)+r(x,j)(m_j-n_j)\\
		m_j+h(x,j)&=n_j+h(x,j)+(m_j-n_j).
	\end{align*}
	Finally, to prove that C.2 is satisfied, we notice that for $j\neq i\in\Theta$, $(x,t)$ and $(y,s)\in U$, $m\in\mathbb{R}^{k}$, and $b>1$
	\begin{align*}
		{G}_j(y,s,&\{m_i:i\in\Theta\},b(t-s),b(x-y),-Y)-{G}_j\left(x,t,\{m_i:i\in\Theta\},b(t-s),b(x-y),X\right)\\
		&=(r(y,j)-r(x,j))m_j-\sum_{i\neq j}\left(\lambda_{j,i}(s)-\lambda_{j,i}(t)\right)(m_i-m_{j})-({\Pi}(y,j)-{\Pi}(x,j))\\
		&+b\left(\alpha(x,j)-\alpha(y,j)\right)\cdot(x-y)+\frac{1}{2}Tr[\sigma\sigma^T(y,j)Y+\sigma\sigma^T(x,j)X]\\
		&\leq L\vert x-y\vert^2+\omega\left(\vert x-y\vert^2+\vert t-s\vert\right),
	\end{align*} 
	for some $L>0$. The last inequality follows in light of the uniform continuity of $\Pi(\cdot,j)$, $r(\cdot,j)$ and $\lambda_{j,i}(\cdot)$ in $U$. The calculations involving $Tr[\sigma\sigma^T(y,j)Y+\sigma\sigma^T(x,j)X]$ may be seen in Example 3.6 of Crandall, Ishii and Lions \cite{crandall1992user}. Finally, the result follows from the uniform continuity of the functions $h(\cdot,i)$ in $U$, for all $i\in\Theta$.
\end{proof}

The next result states that there is a unique solution $v$ to the boundary problem \eqref{Eq HJB}-\eqref{boundary condition}, which is the value function $V^*$. Furthermore, from the proof of Theorem \ref{Verification theorem 2}, one can observe that
\begin{align}
	V^*(x,t,i)&=\begin{cases}
		u(x,t,i),&(x,t,i)\in A_v\\
		-h(x,i),&(x,t,i)\notin A_v
	\end{cases},
\end{align}
where 
\begin{equation}\label{Av}
	A_v=\{(x,t,i)\in I\times[0,\infty)\times\Theta:u(x,t,i)>-h(x,i)\}
\end{equation}
and $u(x,t,i)$ satisfies, in the viscosity sense, the partial differential equation (PDE)
\begin{equation}\label{PDE auxiliary}
	-\tilde{\cal L}v(x,t,i)-\Pi(x,i)=0.
\end{equation}
For future reference, we introduce the stopping time
\begin{equation}\label{tauv}
	\tau_v\equiv\inf\{s\geq 0:(X_s,\zeta_s,\theta_s)\notin  A_v\}. 
\end{equation}
The result presented here follows the idea of Theorem 3.1 of {\O}ksendal and Reikvam \cite{reikvam1998viscosity}. 
\begin{teo}\label{Verification theorem 2}
	Suppose that $v$ is a viscosity solution to the system of equations \eqref{Eq HJB} and satisfies conditions \eqref{boundary condition}. Additionally, assume that 
	\begin{align}\label{UI2}
		\{v(X_\tau,\zeta_\tau,\theta_\tau)\}_{\tau\in {\cal S}}& \text{ is a uniformly integrable family of random variables.} 
	\end{align}
	Then, $v$ is the unique solution to \eqref{Eq HJB}-\eqref{boundary condition} that satisfies \eqref{UI2} and verifies $v=V^*$. Furthermore, $\tau^*=\tau_v$.
\end{teo}
\begin{proof}
	To prove the result, we introduce the following: (i) 
	an open bounded set $A_N\subset I\times[0,\infty)$ such that $A_N\nearrow I\times[0,\infty)$, as $N\to\infty$,
	and (ii) the function $v_N$ that verifies $v_N(x,t,i)=v(x,t,i), \text{ for all  }(x,t,i)\in\overline{A}_N\times\Theta,$  where $v$ is a viscosity solution to \eqref{Eq HJB}, while satisfying conditions \eqref{boundary condition}. By construction, $v_N$ is a solution to \eqref{Auxiliary boundary problem} with $\hat h=v_N$, for all $(x,t,i)\in \overline{A}_N\times\Theta$. Combining Lemma \ref{Lemma conditions unique solution} with the comparison principle in Ishii and Koike \cite{ishii1991viscosity1}, this solution is unique. 
	Consequently, taking into account Proposition \ref{V is a viscosity solution} and Lemma \ref{remark}
	\begin{align*}
		v_N(x,t,i)&=\sup_{\tau\in{\cal S}}E_{x,t,i}\left[\int_{0}^{\tau\wedge\tau_{N}}e^{-\rho_s}\Pi({X}_s,\theta_s)ds+e^{-\rho_{\tau\wedge\tau_{N}}}v_N\left(X_{\tau\wedge\tau_{N}},\zeta_{\tau\wedge\tau_{N}},\theta_{\tau\wedge\tau_{N}}\right)1_{\{\tau\wedge\tau_N<\infty\}}\right],
	\end{align*}
	where $\tau_{N}\equiv\inf\{s>0:(X_s,\zeta_s)\notin A_N\}$. We note that, by construction
	\begin{equation*}
		A_N\nearrow I\times [0,\infty)\quad \text{and}\quad v(x,t,i)=\lim_{N\to\infty}v_N(x,t,i).
	\end{equation*} 
	In additionally, similarly to \eqref{Auxiliary limit 1} and \eqref{Auxiliary limit 2} we can obtain
	\begin{align*}
		\lim_{N\to\infty}E_{x,t,i}\left[\int_{0}^{\tau\wedge\tau_{N}}e^{-\rho_s}\Pi({X}_s,\theta_s)ds\right]=E_{x,t,i}\left[\int_{0}^{\tau\wedge T^I}e^{-\rho_s}\Pi({X}_s,\theta_s)ds\right]
	\end{align*}
	and
	\begin{align*}
		\lim_{N\to\infty}E_{x,i}\left[e^{-\rho_{\tau\wedge\tau_{N}}}v_N(X_{\tau\wedge\tau_{N}},\zeta_{\tau\wedge\tau_{N}},\theta_{\tau\wedge\tau_{N}})\right]=E_{x,i}\left[e^{-\rho_{\tau\wedge T^I}}v(X_{\tau},\zeta_\tau,\theta_{\tau})\right].
	\end{align*}
	Since this holds true for every ${\cal F}_s-$stopping time $\tau$, then
	\begin{align*}
		v(x,i)=&\lim_{N\to+\infty}v_N(x,i)
		\\
		=&\lim_{N\to+\infty}\sup_{\tau\in{\cal S}}E_{x,i}\left[\int_{0}^{\tau\wedge\tau_{N}}e^{-\rho_s}\Pi({X}_s,\theta_s)ds+e^{-\rho_{\tau\wedge\tau_{N}}}v_N(X_{\tau\wedge\tau_{N}},\zeta_{\tau\wedge\tau_{N}},\theta_{\tau\wedge\tau_{N}})1_{\{\tau\wedge\tau_N<\infty\}}\right]\\
		=&\sup_{\tau\in{\cal S}}E_{x,t,i}\left[\int_{0}^{\tau}e^{-\rho_s}\Pi({X}_s,\theta_s)ds+e^{-\rho_\tau}v(X_{\tau},\zeta_\tau,\theta_\tau)1_{\{\tau<\infty\}}\right]\\
		\geq& \sup_{\tau\in{\cal S}}E_{x,t,i}\left[\int_{0}^{\tau}e^{-\rho_s}\Pi({X}_s,\theta_s)ds+e^{-\rho_\tau}h(X_{\tau},\theta_\tau)1_{\{\tau<\infty\}}\right]=V^*(x,t,i),
	\end{align*}
	the last inequality following in light of Lemma \ref{v > h}.
	
	In order to obtain the reverse inequality, we note that, by combining Lemma \ref{v > h} with the first part of this proof, $v_N$ is the unique viscosity solution to \eqref{ODE expected value}-\eqref{ODE terminal condition} in $A_N$ if we replace $\hat{h}$ with $v_N$ and $A_v^i=\{(x,t)\in A_N:v(x,t,i)>h(x,i)\}$. 
	If $A_v^N=\cup_{i\in\Theta}A_v^i=A_v\cap A_N$ and $\tilde{\tau}_N\equiv\inf\{s>0\,:(X_s,\zeta_s,\theta_s)\notin A_v^N\}=\tau_v\cap\tau_N$, then, in view of Lemma \ref{Corollary existence}, it follows that 
	\begin{align*}
		v_N(x,t,i)&=E_{x,t,i}\left[\int_0^{\tilde{\tau}_N}e^{-\rho_s}\Pi(X_s,\theta_s)ds-e^{-\rho_{\tilde{\tau}_N}}v_N(X_{\tilde{\tau}_N},\zeta_{\tilde{\tau}_N},\theta_{\tilde{\tau}_N})1_{\tilde{\tau}_N<\infty}\right].
	\end{align*}
	With a similar argument to the previous one, we obtain
	\begin{align*}
		&v(x,t,i)=\lim_{N\to+\infty}v_N(x,t,i)
		\\
		&=E_{x,t,i}\Bigg[\int_{0}^{\tau_v\wedge T}e^{-\rho_s}\Pi({X}_s,\theta_s)ds+e^{-\rho_{\tau_v\wedge T}}v_N(X_{\tau_v\wedge T},\zeta_{\tau_v\wedge T},\theta_{\tau_v\wedge T})1_{\{\tau_v\wedge T<\infty\}}\Bigg]\\
		&=E_{x,t,i}\left[\int_{0}^{\tau_v\wedge T}e^{-\rho_s}\Pi({X}_s,\theta_s)ds+e^{-\rho_{\tau_v\wedge T}}h(X_{\tau_v\wedge T},\theta_{\tau_v\wedge T})1_{\{\tau_v\wedge T<\infty\}}\right]\\
		&\leq V^*(x,t,i),
	\end{align*}
	which concludes the proof.
\end{proof}

\section{The One-dimensional Case}\label{The One Dimensional Case}
In this section, we present stronger results concerning the optimal stopping problem, when $\theta$ is a homogeneous continuous Markov chain and $X$ is a one dimensional diffusion. To clarify, we make the following assumption:
\begin{assumption}\label{A section 5.1}
	The set $D$ is an interval of the form $]a,b[$ with $-\infty\leq a<b\leq\infty$, $n=m=1$ and the Borel measurable functions $\alpha(\cdot,i):I\to\mathbb{R}$ and  $\sigma(\cdot,i):D\to\mathbb{R}$ are such that the SDE \eqref{Eq SDE}, for each initial condition, has a unique strong solution $(W,X)$ on the filtered probability space $({\Om}, {\fcal}, ({\fcal}_s)_{s\geq 0}, {P})$ that remains in $D$ for all times. Additionally, the process $\theta$ is such that, for every $j,m\in\Theta$,
	\begin{align}
		&P(\nu_n-\nu_{n-1}\leq s\,\vert\, \theta_{\nu_{n-1}}=j)=1-e^{\lambda_j {s}},\quad\text{for all }s\geq 0,\\
		&P(\theta_{\nu_n}=m\,\vert\,\theta_{\nu_{n-1}}=j)=p_{j,m},
	\end{align}
	with $\lambda_j>0$ and $p_{j,m}\in[0,1]$ verifying $\sum_{j\neq m}p_{j,m}=1$ and $p_{jj}=0$. Finally, for every $n_1,n_2\in\mathbb{N}$, the random variables $(\nu_{n_1}-\nu_{n_1-1})\text{ and } (\nu_{n_2}-\nu_{{n_2}-1})$ are independent. Henceforward, we adopt the following notation: $\lambda_{{ j,m}}=\lambda_j\times p_{j,m}$.	
\end{assumption}
As previously mentioned, our main goal is to find the optimal strategy $\tau^*$ which maximizes the function defined in \eqref{Eq expected outcome}, in the interval $I\subseteq D$. Under the previous assumption, we prove that $V^*$, defined in \eqref{Optimal Stopping Problem} is no longer time dependent, which means that, in the homogeneous case, we simply need to use the process $(X,\theta)$. To ensure that the optimal stopping problem is well defined, we make the following assumption:
\begin{assumption}\label{A section 5.2}
	The Borel measurable functions  $\Pi(\cdot,i),h(\cdot,i),r(\cdot,i):D\to\mathbb{R}$, with $i\in \Theta$, are such that: 
	\begin{itemize}
		\item[(1)] Assumption \ref{Assumption : last one} is satisfied;
		\item[(2)] $h(\cdot,i)\in C( D)$;
		\item[(3)] $r(\cdot,i)>0$.
	\end{itemize}	
\end{assumption}
Now, we introduce the differential operator
\begin{align*}
	(\tilde{\cal L}\phi)(x,i)&=-r(x,i)\phi(x,i)+\alpha(x,i)\phi'(x,i)+\frac{1}{2}\sigma^2(x)\phi''(x,i)+\sum_{j\neq i}\lambda_{i,j}\left(\phi(x,j)-{\phi}(x,i)\right),
\end{align*}
where $\phi'(\cdot,i)$ and $\phi''(\cdot,i)$ are respectively the first and the second derivatives of $\phi$ in the first argument.
Additionally, we consider the system of HJB equations
\begin{align}
	\label{Eq HJB ODE}
	&\min\Big\{-(\tilde{\cal L}v)(x,i)-\Pi(x,i),v(x,i)+{h}(x,i)\Big\}=0, \quad\text{for all }(x,i)\in D\times\Theta,
\end{align} 
and the boundary condition 
\begin{align}\label{unideimensional boundary condition}
	v(x)=-h(x) \quad\text{for all }x\in\partial I.
\end{align}
Before we define the concept regarding the solution we consider throughout this section, we will introduce some notation. We denote by $DC(I)$ the set of functions that are the difference of two convex function in $I$. Recall that $f\in DC(I)$ if and only if $f$ is absolutely continuous in $I$ $(f\in AC(I))$ and $f'$ is of bounded variation $(f'\in BV(I))$. Furthermore, if  $f\in DC(I)$, then the left-hand side derivative of $f_{-}'$ exists and its second distributional derivative is a measure. From the Lebesgue's Decomposition Theorem, there are two $\sigma$-finite signed measures, $f''_{ac}(x)dx$ and $f''_{s}(dx)$ such that: 
\begin{itemize}
	\item $f''(dx)=f''_{ac}(x)dx+f''_{s}(dx)$;
	\item $f''_{ac}(x)dx$ is absolutely continuous with respect to the Lebesgue measure, $\mu$;
	\item $f''_{s}(dx)$ and $\mu$ are singular. 
\end{itemize}
\begin{defi}\label{Definition of solution}
	Let $v:\overline{I}\times\Theta\to\mathbb{R}$ be a function such that $v(\cdot,i)\in DC(I)$, for each $i\in\Theta$, and  $\tilde{\cal L}^{ac}$ be the differential operator defined by 
	\begin{align*}
		(\tilde{\cal L}^{ac}v)(x,i)&=-r(x,i)v(x,i)+\alpha(x){v}'_{-}(x,i) +\frac{1}{2}\sigma^2(x){v}_{ac}''(x,i)+\sum_{j\neq i}\lambda_{i,j}\left(\phi(x,j)-{\phi}(x,i)\right).
	\end{align*}
	$v$ is a solution to the system of HJB equations \eqref{Eq HJB ODE} if it satisfies, for every $i\in\Theta$,
	\begin{align}
		&\min\Big\{-(\tilde{\cal L}^{ac}v)(x,i)-\Pi(x,i),v(x,i)+{h}(x,i)\Big\}=0, \quad\mu-\text{almost everywhere in } D.
	\end{align} 
\end{defi}
\begin{remark}\label{Reamrk definition of solution}
	Our analysis will rely on solutions $v$ satisfying Definition \ref{Definition of solution} such that, for each $i\in\Theta$:
	\begin{itemize}
		\item[1)]  $-v''_{s}(dx,i)$ is a positive measure;
		\item[2)] $\supp v_s''(dx,i)\subseteq\{x\in I:v(x,i)=-{h}(x,i)\}$. 
	\end{itemize}
\end{remark}
Note that this definition was already used in the literature of optimal stopping and optimal switching as one can see, for instance, in Lamberton and Zervos \cite{lamberton2013optimal} and Zervos \cite{zervos2003problem}). Henceforward, we will adopt a similar argumentation to the one present at the first aforementioned reference.

Theorem \ref{verification theorem - unidimensional case} provides a general verification result for the optimal stopping problem when $\theta$ is a homogeneous Markov chain and $X$ is a one-dimensional switching diffusion. To prove such a result, we need an appropriate It\^o formula. For one-dimensional semimartingales, the Meyer-It\^o formula is a well-known generalization of the classical It\^o formula that relies on the concept of \textit{Local Time}, which is valid for any function $f\in DC(I)$ (see Protter \cite{protter1990stochastic}). In Lemma \ref{ito lemma}, we provide an appropriate Meyer-It\^o formula for the process $(X,\theta)$. From now on, we will denote by $L^c$ the local time associated with the process $X$ at level $c$, by $A^i$ the process
$${A}_t^{i}=\frac{1}{2}\int_D{\phi}_{s}''(dc,i)L_t^c,\quad  \text{for } i=0,1,$$
and by ${\cal A}^{ac}$ the operator
$$
({\cal A}^{ac}\phi)(x,i)\equiv-r(x)\phi(x,i)+\alpha(x)\phi_{-}'(x,i)+\frac{1}{2}\sigma^2(x)\phi_{ac}''(x,i),
$$
where $\phi$ is such that $\phi(\cdot,i)\in DC(I)$, for each $i=0,1$. 

\begin{lemma}\label{ito lemma}
	Let $({X},\theta)$ be the $(1+1)$-dimensional process defined by Equations \eqref{theta} and \eqref{Eq SDE}, taking into account Assumptions \ref{A section 5.1}. Furthermore, assume that $\phi:D\times\Theta\to\mathbb{R}$ is such that $\phi(\cdot,i)\in DC(D)$, with $i\in\Theta$. Then, for every $t\in[0,T^D]$
	\begin{align}
		E_{x,i}\left[e^{-\rho_t}\phi({X}_t,\theta_t)\right]&=\phi(x,i)+E_{x,i}\left[\int_0^{t}e^{-\rho_s}(\tilde{\cal L}^{ac})\phi(X_s,\theta_s)ds\right]+E_{x,i}\left[\sum_{j=1}^k\int_{0}^{t}\nonumber e^{-\rho_s}1_{\{\theta_s=j\}}d{A}_s^{j}\right]\\
		&+E_{x,i}\left[\int_{0}^{t}{\phi'}_{-}(X_s,i)\sigma(X_s,i)dW_s\right].
	\end{align}
\end{lemma}
\begin{proof}	
	Let  $\nu_n$, with $n\in\mathbb{N}$, be defined as in Section \ref{The stochastic process}, and assume that $t\in[\nu_n,\nu_{n+1}[$, then $\phi$ admits the following decomposition
	\begin{align*}
		e^{-\rho_t}\phi(X_t,\theta_t)=&e^{-\rho_t}\phi(X_t,\theta_t)-e^{-\rho_{\nu_n}}\phi(X_{\nu_n},\theta_{\nu_n})
		+\sum_{j=1}^n\left(e^{-\rho_{\nu_{j}}}\phi(X_{\nu_{j}},\theta_{\nu_{j}})-e^{-\rho_{\nu_{j}^-}}\phi(X_{\nu_{j}^-},\theta_{\nu_{j}^-})\right)\\
		+&\sum_{j=2}^n\left(e^{-\rho_{\nu_j^-}}\phi(X_{\nu_j^-},\theta_{\nu_j^-})-e^{-\rho_{\nu_{j-1}}}\phi(X_{\nu_{j-1}},\theta_{\nu_{j-1}})\right)
		+ e^{-\rho_{\nu_1^-}}\phi(X_{\nu_1^-},\theta_{\nu_1^-}).
	\end{align*}
	Therefore, for any $t\geq 0$,
	\begin{align*}
		e^{-\rho}\phi(X_t,\theta_t)=&e^{-\rho_{t\wedge\nu_1^-}}\phi(X_{t\wedge\nu_1^-},i)+\sum_{j=1}^\infty\Big(e^{-\rho_{\nu_j}}\phi(X_{\nu_j},\theta_{\nu_j})-e^{-\rho_{t\wedge\nu_{j}^-}}\phi(X_{t\wedge\nu_{j}^-},\theta_{t\wedge\nu_{j}^-})\Big)1_{\{ t\geq\nu_{j}\}}\\
		+&\sum_{j=2}^\infty\left(e^{-\rho_{t\wedge\nu_j^-}}\phi(X_{t\wedge\nu_j^-},\theta_{t\wedge\nu_j^-})-e^{-\rho_{\nu_{j-1}}}\phi(X_{\nu_{j-1}},\theta_{t\wedge\nu_{j-1}})\right)1_{\{ t\geq\nu_{j-1}\}}.	
	\end{align*}
	Since $\theta_t$ is constant for $\nu_n\leq t<\nu_{n+1}$, we can say, without loss of generality, that $\theta_t=i$. Then, from {the Meyer}-It\^o Formula (see Theorem IV.70 in Protter \cite{protter1990stochastic}), we get
	\begin{align*}
		\phi(X_{t},i)-\phi(X_{\nu_n},i)=&\int_{\nu_n}^{t}\alpha(X_s,i){\phi'}_-(X_s,i)ds+\frac{1}{2}\int_D\phi''(dc,i)L_t^c+\int_0^{t}{\phi'}_{-}(X_s,i)\sigma(X_s,i)dW_s,
	\end{align*} 
	where $\phi$ is such that $\phi(\cdot,i)\in DC(D)$.
	From the Occupation Times Formula we obtain
	$$
	\int_D{\phi}_{ac}''(dc,i)L_t^c=\int_{\nu_n}^t\sigma^2(X_s,i)\phi''_{ac}(X_s,i)ds,
	$$
	which allows us to write
	\begin{align*}
		\phi(X_t,i)-\phi(X_{\nu_n},i)=&\int_{\nu_n}^{t}\alpha(X_s){\phi'}_-(X_s,i)+\frac{1}{2}\sigma^2(X_s,i)\phi_{ac}''(X_s,i)ds\nonumber+\frac{1}{2}\int_D\phi_{s}''(dc,i)L_t^c\\
		+&\int_{\nu_n}^{t}{\phi'}_{-}(X_s,i)\sigma(X_s,i)dW_s.
	\end{align*}
	By using the integration by parts, we have
	\begin{align}\nonumber
		e^{-\rho_t}\phi(X_t,i)-e^{-\rho_{\nu_n}}\phi(X_{\nu_n},i)=&\int_{\nu_n}^{t}e^{-\rho_s}({\cal A}^{ac}\phi)(X_s,i)ds+\int_{\nu_n}^te^{-\rho_s}dA_s^i\\
		+&\int_{\nu_n}^{t}e^{-\rho_s}{\phi'}_{-}(X_s,i)\sigma(X_s)dW_s.\label{Mayer-Ito Formula continuous semimartingales}
	\end{align}
	Taking into account that this argument remains valid in any interval $[\nu_n,\nu_{n+1}[$, with $n\in \mathbb{N}$, we obtain that, for every $0<t\leq T^D$
	\begin{align*}
		e^{-\rho_t}\phi({X}_t,\theta_t)=&\phi(x,i)+\int_0^{t}e^{-\rho_t}({\cal A}^{ac}\phi)(X_s,\theta_s)ds+\sum_{j=1}^k\int_{0}^{t} e^{-\rho_s}1_{\{\theta_t=j\}}d{A}_s^{j}\\
		+&\sum_{j=1}^\infty\Big(e^{-\rho_{\nu_j}}\phi(X_{\nu_j},\theta_{\nu_j})-e^{-\rho_{t\wedge\nu_{j}^-}}\phi(X_{t\wedge\nu_{j}^-},\theta_{t\wedge\nu_{j}^-})\Big)1_{\{ s\geq\nu_{j}\}}\\
		+&\int_0^{t}e^{-\rho_s}{\phi'}_{-}(X_s,\theta_s)\sigma(X_s,\theta_s)dW_s.
		\nonumber
	\end{align*}
	Now, by using a similar argument to the one used in \eqref{jumps decomposition}, we obtain the result.
\end{proof}

\begin{teo}\label{verification theorem - unidimensional case}
	Let $V^*$ be the value function defined as in \eqref{Optimal Stopping Problem}, taking into account Assumptions \ref{A section 5.1} and \ref{A section 5.2}. Assume that there is a function $v:\overline{I}\times\Theta\to\mathbb{R}$ such that $v(\cdot,i)\in DC(I)$ and $v$ is a solution to the system of HJB equations  \eqref{Eq HJB ODE} in the sense of Definition \ref{Definition of solution} and the process
	\begin{equation}\label{martingale}
		{\left\{\int_{0}^{t}{\phi'}_{-}(X_s,i)\sigma(X_s,i)dW_s\right\}}_{t\geq 0}\quad\text{is a martingale}.
	\end{equation}
	Furthermore, assume that $v$ is such that 1) in Remark \ref{Reamrk definition of solution} is fulfilled. The following statements are true:
	\begin{itemize}
		\item[1)]$v(x,i)\geq J(x,i,\tau)$, for all $\tau\in {\cal S}$;
		\item[2)]additionally, if $v$ is such that statement 2) in Remark \ref{Reamrk definition of solution} holds true, the boundary condition \eqref{unideimensional boundary condition} is satisfied and
		\begin{equation}\label{growth conditon}
			\{v(X_\tau,\theta_\tau)\}_{\tau\in {\cal S}} \text{ is a uniformly integrable family of random variables.} 
		\end{equation}
		Then, $V^*=v$, $\tau^*=\inf\{s\geq 0:v(X_s,\theta_s)\leq -{h}(X_s,\theta_s)\}$ is the optimal strategy.
	\end{itemize}
\end{teo}
\begin{proof}
	We start by proving statement 1) of Theorem \ref{verification theorem - unidimensional case}. Fix $\tau\in {\cal S}$ and let $\{\tau_n\}_{n\in\mathbb{N}}\subset {\cal S}$ be an increasing sequence, such that $\tau_n\nearrow\tau$. Then, by using Lemma \ref{ito lemma} and condition \eqref{martingale}, we obtain
	\begin{align*}
		J(x,t,i,\tau_n)=&E_{x,t,i}\left[\int_{0}^{\tau_n}e^{-\rho_s}\Pi(X_s,\theta_s)ds-e^{-\rho_{\tau_n}}h(X_{\tau_n},\theta_{\tau_n})\right]\\
		=&v(x,i)+E_{x,t,i}\left[\int_{0}^{\tau_n}e^{-\rho_s}\left(\Pi(X_s,\theta_s)+(\tilde{\cal L}^{ac})v(X_s,\theta_s)\right)ds\right]\\
		-&E_{x,t,i}\left[e^{-\rho_{\tau_n}}\left(v(X_{\tau_n},\theta_{\tau_n})+h(X_{\tau_n},\theta_{\tau_n})\right)\right]+E_{x,t,i}\left[\sum_{j=1}^k\int_{0}^{\tau_n}\nonumber e^{-\rho_s}1_{\{\theta_s=j\}}d{A}_s^{j}\right].
	\end{align*}
	In light of Definition \ref{Definition of solution}, we get
	\begin{align*}
		J(x,t,i,\tau_n)\leq
		&v(x,i)+E_{x,t,i}\left[\sum_{j=1}^k\int_{0}^{t}\nonumber e^{-\rho_s}1_{\{\theta_s=j\}}d{A}_s^{j}\right].
	\end{align*}
	To proceed, we note that the local time associated with the process $X$ at level $c$, $L^c$, is increasing and c\`adlag. Therefore,
	\begin{align}\label{dA negativo}
		d{A}_s^{i}=\frac{1}{2}d\int_D{v}_{s}''(dc,i)L_t^c=\frac{1}{2}d\int_{\supp_i}{v}_{s}''(dc,i)L_t^c\leq 0,
	\end{align}
	and, consequently, $J(x,t,i,\tau_n)\leq v(x,i)$. With a similar argument to \eqref{Auxiliary limit 1} and \eqref{Auxiliary limit 2}, we have that
	$$
	\lim_{n\to\infty}J(x,t,i,\tau_n)=J(x,t,i,\tau)\leq v(x,i),\quad\text{for all }\tau\in{\cal S}.
	$$
	To prove statement 2), we consider the stopping time $\tau_0\in{\cal S}$ given by $\tau_0 \equiv\inf\{s\geq 0:v(X_s,\theta_s)\leq -{h}(X_s,\theta_s)\}$ and $\{\tau_n\}_{n\in\mathbb{N}}\subset {\cal S}$, which is an increasing sequence verifying $\tau_n\nearrow T^I$. Then, we obtain from Lemma \ref{ito lemma} that
	\begin{align}
		e^{-\rho_{\tau_0\wedge\tau_n}}&v({X}_{\tau_0\wedge\tau_n},\theta_{\tau_0\wedge\tau_n})=v(x,i)+\int_0^{{\tau_0\wedge\tau_n}}e^{-\rho_s}(\tilde{\cal L}^{ac})v(X_s,\theta_s)ds\nonumber\\
		+&\sum_{j=1}^k\int_{0}^{\tau_0\wedge\tau_n}\nonumber e^{-\rho_s}1_{\{\theta_s=j\}}d{A}_s^{j}+\int_{0}^{\tau_0\wedge\tau_n}{v'}_{-}(X_s,i)\sigma(X_s,i)dW_s.
	\end{align}
	Consequently, taking into account that $\tau_0=\tau_0\wedge T^I$ and the boundary problem \eqref{Eq HJB ODE}-\eqref{unideimensional boundary condition} is satisfied (in the sense of Definition \ref{Definition of solution})
	\begin{align}
		\int_0^{{\tau_0\wedge\tau_n}}e^{-\rho_s}\Pi(X_s,\theta_s)ds&-e^{-\rho_{\tau_0}}h({X}_{\tau_0},\theta_{\tau_0})1_{\{\tau_0\leq\tau_n\}}=v(x,i)\\
		&-e^{-\rho_{\tau_n}}v({X}_{\tau_n},\theta_{\tau_n})1_{\{\tau_0>\tau_n\}}\nonumber
		+\sum_{j=1}^k\int_{0}^{\tau_0\wedge\tau_n}\nonumber e^{-\rho_s}1_{\{\theta_s=j\}}d{A}_s^{j}\\
		&+\int_{0}^{\tau_0\wedge\tau_n}{v'}_{-}(X_s,i)\sigma(X_s,i)dW_s.\nonumber
	\end{align}
	Assuming that $\theta_s=j$, with $j\in\Theta$ and for every $s\in[\nu_n\wedge\tau_0,\nu_{n+1}\wedge\tau_0[$, then, by combining the definition of $\tau_0$, Equation \eqref{dA negativo}, and statement 2) of Remark \ref{Reamrk definition of solution}, we get
	\begin{align*}
		d{A}_s^{i}= 0.
	\end{align*}
	Thus, from condition \eqref{martingale}, we obtain
	\begin{align*}
		&E_{x,t,i}\left[\int_0^{{\tau_0\wedge\tau_n}}e^{-\rho_s}\Pi(X_s,\theta_s)ds-e^{-\rho_{\tau_0}}h({X}_{\tau_0},\theta_{\tau_0})1_{\{\tau_0\leq\tau_n\}}\right]=v(x,i)\\
		&~~~~~~~~~~~~~~~~~~~~~~~~~~~~~~~~~~~~~~~~~~~~~~~~~~~~~~~~~~~~~~~~~~~~-E_{x,t,i}\left[e^{-\rho_{\tau_n}}v({X}_{\tau_n},\theta_{\tau_n})1_{\{\tau_0>\tau_n\}}\right].
	\end{align*}
	Consequently, with a similar argument to the one used in \eqref{Auxiliary limit 1}, we obtain
	$$
	\lim_{n\to\infty}E_{x,t,i}\left[\int_0^{{\tau_0\wedge\tau_n}}e^{-\rho_s}\Pi(X_s,\theta_s)ds\right]=E_{x,t,i}\left[\int_0^{{\tau_0\wedge T^I}}e^{-\rho_s}\Pi(X_s,\theta_s)ds\right].
	$$
	Additionally, $\{e^{-\rho_{\tau_0}}h({X}_{\tau_0},\theta_{\tau_0})1_{\{\tau_0\leq\tau\}}\}_{\tau\in{\cal S}}$ is a uniformly integrable family of random variables and, consequently,
	\begin{equation}\label{uniform integrability h}
		\lim_{n\to\infty}E_{x,t,i}\left[e^{-\rho_{\tau_0}}h({X}_{\tau_0},\theta_{\tau_0})1_{\{\tau_0\leq\tau_n\}}\right]=E_{x,t,i}\left[e^{-\rho_{\tau_0}}h({X}_{\tau_0},\theta_{\tau_0})1_{\{\tau_0\leq T^I\wedge\infty\}}\right].
	\end{equation}
	From Assumption \ref{A section 5.2}, $0\leq e^{-\rho_{\tau}}<1$  for every $\tau\in{\cal S}$, and, accordingly,  \linebreak$\{e^{-\rho_{\tau}}v(X_{\tau},\theta_{\tau})1_{\tau_0>\tau}\}_{\tau\in{\cal S}}$ is a uniformly integrable family of random variables. If $T^I=\infty$, then $e^{-\rho_{\tau_n}}\to 0$ and, consequently, $e^{-\rho_{\tau_n}}v(X_{\tau_n},\theta_{\tau_n})1_{\tau_0>\tau_n}\to 0$, $P-$almost surely. Additionally, if $T^I<\infty$, then  $ e^{-\rho_{\tau_n}}v(X_{\tau_n},\theta_{\tau_n})1_{\tau_0>\tau_n}\to -e^{-\rho_{T^I}}h(X_{T^I},\theta_{T^I})1_{\tau_0>T^I}=0$, $P-$almost surely.  Therefore, if condition \eqref{growth conditon} holds true, we get \begin{equation*}
		J(x,t,i,\tau^*)=V^*(x,i)=v(x,i).
	\end{equation*}
	
\end{proof}
Before we finish this section, we note that, if we relax the assumption $r(.,i)>0$ for every $i\in\Theta$, the condition \eqref{growth conditon} may not be, in general, sufficient to keep the result true. In this case, a condition like 
\begin{equation*}
	\lim_{n}E_{x,i}[e^{-\rho_{\tau_n}}\vert v(X_{\tau_n},\theta_{\tau_n})\vert]=0, \text{ with }\tau_n\nearrow T^I,
\end{equation*}
would be required.

\section*{Acknowledgements}
Funding: This work was supported by the Funda\c{c}\~ao para a Ci\^encia e Tecnologia (FCT) [grant number SFRH/BD/102186/2014].

\section*{References}

\bibliography{myrefs}

\begin{thebibliography}{10}
\expandafter\ifx\csname url\endcsname\relax
  \def\url#1{\texttt{#1}}\fi
\expandafter\ifx\csname urlprefix\endcsname\relax\def\urlprefix{URL }\fi
\expandafter\ifx\csname href\endcsname\relax
  \def\href#1#2{#2} \def\path#1{#1}\fi

\bibitem{dixit1994investment}
A.~K. Dixit, R.~S. Pindyck, Investment under uncertainty, Princeton university
  press, 1994.

\bibitem{trigeorgis1996real}
L.~Trigeorgis, Real options: Managerial flexibility and strategy in resource
  allocation, MIT press, 1996.

\bibitem{boomsma2012renewable}
T.~K. Boomsma, N.~Meade, S.-E. Fleten, Renewable energy investments under
  different support schemes: A real options approach, European Journal of
  Operational Research 220~(1) (2012) 225--237.

\bibitem{boomsma2015market}
T.~K. Boomsma, K.~Linnerud, Market and policy risk under different renewable
  electricity support schemes, Energy 89 (2015) 435--448.

\bibitem{adkins2016subsidies}
R.~Adkins, D.~Paxson, Subsidies for renewable energy facilities under
  uncertainty, The Manchester School 84~(2) (2016) 222--250.

\bibitem{fleten2016green}
S.-E. Fleten, K.~Linnerud, P.~Moln{\'a}r, M.~T. Nygaard, Green electricity
  investment timing in practice: Real options or net present value?, Energy 116
  (2016) 498--506.

\bibitem{kitzing2017real}
L.~Kitzing, N.~Juul, M.~Drud, T.~K. Boomsma, A real options approach to analyse
  wind energy investments under different support schemes, Applied Energy 188
  (2017) 83--96.

\bibitem{guerra2017hysteresis}
M.~Guerra, P.~Kort, C.~Nunes, C.~Oliveira, Hysteresis due to irreversible exit:
  Addressing the option to mothball, Working Paper.

\bibitem{eloe2008optimal}
P.~Eloe, R.~Liu, M.~Yatsuki, G.~Yin, Q.~Zhang, Optimal selling rules in a
  regime-switching exponential gaussian diffusion model, SIAM Journal on
  Applied Mathematics 69~(3) (2008) 810--829.

\bibitem{guo2001explicit}
X.~Guo, An explicit solution to an optimal stopping problem with regime
  switching, Journal of Applied Probability 38~(2) (2001) 464--481.

\bibitem{guo2004closed}
X.~Guo, Q.~Zhang, Closed-form solutions for perpetual american put options with
  regime switching, SIAM Journal on Applied Mathematics 64~(6) (2004)
  2034--2049.

\bibitem{pemy2014optimal}
M.~Pemy, Optimal stopping of markov switching l{\'e}vy processes, Stochastics
  An International Journal of Probability and Stochastic Processes 86~(2)
  (2014) 341--369.

\bibitem{pemy2006optimal}
M.~Pemy, Q.~Zhang, Optimal stock liquidation in a regime switching model with
  finite time horizon, Journal of Mathematical Analysis and Applications
  321~(2) (2006) 537--552.

\bibitem{liu2016optimal}
R.~Liu, Optimal stopping of switching diffusions with state dependent switching
  rates, Stochastics 88~(4) (2016) 586--605.

\bibitem{egami2017optimal}
M.~Egami, R.~Kevkhishvili, On the optimal stopping problem of linear diffusions
  in regime-switching models, arXiv preprint arXiv:1711.08883.

\bibitem{mao2006stochastic}
X.~Mao, C.~Yuan, Stochastic differential equations with Markovian switching,
  Imperial College Press, 2006.

\bibitem{yin2010hybrid}
G.~Yin, C.~Zhu, Hybrid switching diffusions: properties and applications,
  Vol.~63, Springer New York, 2010.

\bibitem{kallenberg2006foundations}
O.~Kallenberg, Foundations of modern probability, Springer Science \& Business
  Media, 2006.

\bibitem{karatzas2012brownian}
I.~Karatzas, S.~Shreve, Brownian motion and stochastic calculus, Vol. 113,
  Springer Science \& Business Media, 2012.

\bibitem{krylov2008controlled}
N.~V. Krylov, Controlled diffusion processes, Vol.~14, Springer Science \&
  Business Media, 2008.

\bibitem{oksendal1998stochastic}
B.~{\O}ksendal, Stochastic differential equations, in: Stochastic differential
  equations, Springer, 1998, pp. 61--78.

\bibitem{bouchard2011weak}
B.~Bouchard, N.~Touzi, Weak dynamic programming principle for viscosity
  solutions, SIAM Journal on Control and Optimization 49~(3) (2011) 948--962.

\bibitem{protter1990stochastic}
P.~Protter, Stochastic differential equations, in: Stochastic Integration and
  Differential Equations, Springer, 1990, pp. 187--284.

\bibitem{ehrhardt2017novel}
M.~Ehrhardt, M.~G{\"u}nther, E.~J.~W. ter Maten, Novel methods in computational
  finance, Vol.~25, Springer, 2017.

\bibitem{touzi2012optimal}
N.~Touzi, Optimal stochastic control, stochastic target problems, and backward
  SDE, Vol.~29, Springer Science \& Business Media, 2012.

\bibitem{ishii1991viscosity1}
H.~Ishii, S.~Koike, Viscosity solutions for monotone systems of second--order
  elliptic pdes, Communications in partial differential equations 16~(6-7)
  (1991) 1095--1128.

\bibitem{crandall1992user}
M.~G. Crandall, H.~Ishii, P.-L. Lions, User’s guide to viscosity solutions of
  second order partial differential equations, Bulletin of the American
  Mathematical Society 27~(1) (1992) 1--67.

\bibitem{reikvam1998viscosity}
B.~{\O}ksendal, K.~Reikvam, Viscosity solutions of optimal stopping problems,
  Stochastics and Stochastic Reports 62~(3-4) (1998) 285--301.

\bibitem{lamberton2013optimal}
D.~Lamberton, M.~Zervos, et~al., On the optimal stopping of a one-dimensional
  diffusion, Electronic Journal of Probability 18.

\bibitem{zervos2003problem}
M.~Zervos, A problem of sequential entry and exit decisions combined with
  discretionary stopping, SIAM Journal on Control and Optimization 42~(2)
  (2003) 397--421.

\end{thebibliography}

\end{document}